\documentclass[reqno,10pt]{amsart}

\numberwithin{equation}{section}

\usepackage{latexsym}
\usepackage{amsmath}
\usepackage{mathtools}
\usepackage{amssymb}
\usepackage{mathrsfs}
\usepackage{graphicx,colordvi}
	\usepackage{upgreek}
\usepackage{ifthen}
\usepackage{color}
\usepackage{graphicx}
\usepackage{bm}

\usepackage[T1]{fontenc}

\numberwithin{equation}{section}

\newtheorem{defi}{Definition}[section]
\newtheorem{theorem}[defi]{Theorem}
\newtheorem{lemma}[defi]{Lemma}

\newtheorem{proposition}[defi]{Proposition}
\newtheorem{remark}[defi]{Remark}

\usepackage{latexsym}
\usepackage{amsmath}
\usepackage{amssymb}

\newcommand{\cA}{{\mathcal A}}

\newcommand{\cD}{{\mathcal D}}
\newcommand{\cE}{{\mathcal E}}
\newcommand{\cF}{{\mathcal F}}

\newcommand{\cB}{{\mathcal B}}

\newcommand{\cL}{{\mathcal L}}
\newcommand{\cM}{{\mathcal M}}
\newcommand{\cR}{{\mathcal R}}

\newcommand{\cT}{{\mathcal T}}

\newcommand{\cU}{{\mathcal U}}
\newcommand{\cV}{{\mathcal V}}

\newcommand{\DD}{{\mathbb D}}
\newcommand{\EE}{{\mathbb E}}
\newcommand{\FF}{{\mathbb F}}

\newcommand{\LL}{{\mathbb L}}
\newcommand{\MM}{{\mathbb M}}
\newcommand{\N}{{\mathbb N}}

\newcommand{\R}{{\mathbb R}}

\newcommand{\PP}{{\mathbb P}}

\newcommand{\ve}{{\varepsilon}}

\newcommand{\sA}{\mathsf{A}}
\newcommand{\sB}{\mathsf{B}}

\newcommand{\sF}{\mathsf{F}}

\newcommand{\sR}{\mathsf{R}}
\newcommand{\sT}{\mathsf{T}}

\newcommand{\bc}{{\bf c}}
\newcommand{\fB}{\mathfrak{B}}
\newcommand{\fD}{\mathfrak{D}}
\newcommand{\fF}{F}  
\newcommand{\fG}{G} 

\newcommand{\e}{\mathrm{e}}

\renewcommand{\epsilon}{\varepsilon}
\newcommand{\eps}{\varepsilon}

\newcommand{\pa}{\partial}
\newcommand{\od}{\mathring{d}}

\renewcommand{\hat}{\widehat}
\newcommand{\wt}{\widetilde}

\newcommand{\na}{\nabla}
\newcommand{\be}{\begin{equation}}
\newcommand{\ee}{\end{equation}}
\newcommand{\bal}{\begin{aligned}}
\newcommand{\eal}{\end{aligned}}
\newcommand{\ba}{\begin{array}{l}}
\newcommand{\ea}{\end{array}}

\newcommand{\HR}[1]{\textcolor{blue}{#1}}

\begin{document}
\title{On nematic electrolytes}

\author{Hengrong Du}
\address{Department of Mathematics and Computer Sciences\\
	Fisk University \\
	Nashville, Tennessee, USA}
\email{hdu@fisk.edu}

\author{Fizay-Noah Lee}
\address{Department of Mathematics\\
        Vanderbilt University\\
        Nashville, Tennessee\\
        USA}
\email{fizay-noah.lee@vanderbilt.edu}

\author{Gieri Simonett}
\address{Department of Mathematics\\
        Vanderbilt University\\
        Nashville, Tennessee\\
        USA}
\email{gieri.simonett@vanderbilt.edu}

\thanks{This work was supported by a grant from the Simons Foundation (\#853237, Gieri Simonett).}




\begin{abstract}
We study a system of nonlinear partial differential equations modeling the electrokinetics of a nematic electrolyte material consisting of various ion species suspended in a nematic liquid crystal within a bounded domain in two or three dimensions. The system couples a Nernst-Planck model for ion concentrations with the Poisson equation for the electrostatic potential, a Navier-Stokes equation for the fluid solvent, and the Ericksen-Leslie equations with general Leslie stress for nematic liquid crystals. We consider the case of isotropic elasticity for the liquid crystal and impose a unit-length constraint on the director field. The no-flux condition for the electrochemical potential leads to a nonlinear (and nonlocal) boundary condition for the ion concentrations. Using the theory of maximal regularity, we prove existence and uniqueness of strong solutions, provide criteria for global existence, and characterize the set of equilibria.
\end{abstract}

\maketitle
\section{Introduction}
\noindent
In this paper we consider a modified version of the system derived in \cite[(2.51)-(2.55)]{CGLW16}
 describing the electrokinetic behavior of a nematic electrolyte consisting of ions that diffuse and are advected within a nematic liquid crystal environment.
The system can be written in terms of the following variables:
\begin{itemize}
\item the concentrations $c_k$,  $k = 1, \ldots  ,N$, with valences $z_k\in \R$,
\item  the electrostatic potential $\Phi$,
\item the macroscopic velocity $v$ of the liquid crystal molecules,
\item the pressure $\pi$,
\item the direction vector $d$, modeling the local orientation of the nematic liquid crystal molecules.
\end{itemize}

 Electrolytes are substances that dissociate into ions when dissolved in a solvent, typically water. They are central to a wide range of physical, biological, and engineering systems. From nerve signal transmission in living organisms to energy storage in batteries, the behavior of electrolytes is governed by coupled physical and chemical phenomena.

 Nematic electrolytes are emerging soft materials that blend the anisotropic, orientational order of nematic liquid crystals with the ionic conductivity characteristic of electrolytes.  This hybrid material system exhibits rich physical behaviors with applications in ionotronics, responsive optics, and electrochemical devices.

 Understanding these systems requires a framework that blends electrostatics, fluid dynamics, and liquid crystals. 
 This framework is encapsulated in a set of fundamental equations: the \textbf{Nernst-Planck equations}, the \textbf{Poisson equation}, the \textbf{Navier-Stokes equations}, and the \textbf{Ericksen-Leslie equations}. 
The interplay between molecular alignment and ionic motion leads to a complex coupled system of equations.

\noindent
 In fact, following the approach in \cite{FRSZ21} we consider a modified version of the system introduced in~\cite{CGLW16}, 
incorporating some simplifications that are commonly used in the mathematical literature on liquid crystals. 
Here, we consider the following system of coupled equations.
Suppose $\Omega$ is a bounded and smooth domain in $\R^n$, $n=2, 3$, with outer unit normal field denoted by $\nu$.
The equations read
\begin{equation}
  \label{main-PDE}
	\left\{
  \begin{aligned}
    \pa_t c_k+  v \cdot  \nabla c_k &= {\rm div}\,(c_k \cD_k \nabla \mu_k),\quad k=1,\ldots,N, &&\text{in} &&\Omega, \\ 
   \nu\cdot ( c_k \cD_k \nabla \mu_k )  &=0, \quad k=1,\ldots, N,&&\text{on} && \partial \Omega, \\
    -{\rm div}\, (\ve(d)\nabla \Phi)&=\Sigma_{k=1}^{N} z_k c_k  &&\text{in} && \Omega,\\
       \Phi &= 0 &&\text{on} && \partial \Omega, \\
    \pa_t v +  v\cdot \nabla  v  +\nabla \pi&=- {\rm div}(\nabla d\odot \nabla d)   + {\rm div}\, \sigma_L 
    + {\rm div} \left((\nabla \Phi\otimes \nabla \Phi)\ve(d) \right)  &&\text{in} && \Omega, \\
    {\rm div}\, v &=0  &&\text{in} && \Omega, \\
     v &= 0 &&\text{on} && \partial \Omega, \\      
    \gamma_1\od+\gamma_2 P(d) D(v)d&= \Delta d+|\nabla d|^2 d   +\ve_a P(d) (\nabla \Phi\otimes \nabla\Phi)d &&\text{in} && \Omega,\\
    |d| & = 1 &&\text{in} && \Omega, \\
    \partial_\nu d &=0 &&\text{on} && \partial \Omega, \\
  (c_k(0), v(0), d(0)) & =  (c^0_k , v^0, d^0) &&\text{in} && \Omega.\\
 \end{aligned}
  \right.
\end{equation}
The quantities appearing in  \eqref{main-PDE} are:
{\begin{itemize}
  \item the ionic \emph{diffusion matrices} $\cD_k$, assumed to be positive definite;
  that is, there exists a constant $\alpha>0$ such that
  \begin{equation}
  \label{D-positive-definite}
  \cD_k \xi \cdot \xi \ge \alpha | \xi |^2,\quad \xi\in \R^n,\quad k=1,\ldots, N,
  \end{equation}

  \item  the electrochemical potentials  $\mu_k$ 
associated to the various ions species $c_k$, given by 
 $$\mu_k:=\ln c_k +z_k \Phi, \qquad k=1,\ldots, N,$$
with $z_k$ the respective valences,

\vspace{2mm}  \item the \emph{dielectric permittivity matrix} $\ve(d)$, given by
  $$\ve(d):=\ve_{\|} d\otimes d+\ve_{\perp}({\rm I}-d\otimes d)=\ve_{\perp}{\rm I} +\ve_a d\otimes d,\qquad \ve_a = \ve_{\|} -\ve_{\perp},$$
  where the constants
  $\ve_{\|}$ and $\ve_{\perp}$ represent the parallel and perpendicular dielectric permittivity with respect to $d$, 
  and where we assume 
  \begin{equation}
  \label{permittivity}
  \ve_{\perp} >0,\qquad \ve_a\ge 0,
  \end{equation}
  \item  
 the Leslie stress tensor $\sigma_L$, given by 
  \begin{equation}
   \label{Leslie-stess}
   \begin{aligned}
  	\sigma_L=\sigma_L(v,d) &:= \alpha_1(D(v)d\cdot d)d\otimes d+\alpha_2 \od\otimes d+\alpha_3 d\otimes \od+\alpha_4 D(v) \\
	&\quad 
	+\alpha_5 D(v)d\otimes d+\alpha_6 d\otimes D(v)d,
 \end{aligned}
  \end{equation}
 where  the Leslie coefficients $\alpha_1, \ldots, \alpha_6$ are material constants.
Here, $\alpha_4 D(v)$ represents the classical Newtonian stress tensor, while the other terms
represent the additional stress produced by the interaction of the liquid crystal molecules,
  
\vspace{2mm} \item
   the co-rotational derivative $\od $ of $d$, given by 
   \begin{equation}
   \label{circ-d}
    \od =\od(v):= \pa_t d+ v \cdot \nabla d-\Omega(v)d,
    \end{equation}
   
\item
   the symmetric and anti-symmetric parts of the velocity gradient, given by
    $$D(v):=\frac{1}{2}(\nabla v+(\nabla v)^{\top})  \quad \text{and}\quad \Omega(v):=\frac{1}{2}(\nabla v-(\nabla v)^{\top}),$$

\item 
the orthogonal projection of $\R^n$ onto ${\rm span}\{d^{\perp}\}$, given by
\begin{equation*}
P(d) := {\rm I} - d \otimes d,
\end{equation*}

\item
the Ericksen stress tensor $\nabla d \odot \nabla d$, given by
$$[\nabla d \odot \nabla d]_{ij} = \partial_i d \cdot \partial_j d,\quad 1\le i,j \le n.$$

\item
Finally, we have $[a\otimes b]_{ij} = a_ib_j$  for vectors $a,b\in \R^n$.
\end{itemize}
We have set various physical constants equal to 1 that would otherwise appear in \eqref{main-PDE}.

\medskip
The  no-flux (blocking) boundary conditions 
 $\nu\cdot {\rm tr}_{\pa\Omega}(c_k \cD_k \nabla \mu_k) =0 $ for the ion concentrations
 correspond to the property that ions cannot cross the boundary $\partial\Omega$
and, thus, their total concentrations remain constant in time.
 This boundary condition can  be restated as
 \begin{equation}
 \label{BCk}
 B_k(\textbf{c}, d)c_k := \nu \cdot {\rm tr}_{\partial\Omega} (\cD_k \nabla c_k) + \nu\cdot {\rm tr}_{\partial\Omega} (z_k\cD_k \nabla \Phi (\textbf{c}, d ))c_k=0,
\end{equation}
with
 $\bc= (c_1,\ldots ,c_N)$ and $k=1,\ldots, N$,
where $\Phi=\Phi(\bc, d)$ is the solution of the elliptic problem
\begin{equation}
\label{equation-phi}
\left\{
\begin{aligned}
-{\rm div}\, (\ve(d)\nabla \Phi)&=\sum_{k=1}^{N} z_k c_k  &&\text{in} &&\Omega,\\
       \Phi &= 0 &&\text{on} &&\partial \Omega. \\
\end{aligned}
\right.
\end{equation}
We note that this results in a nonlinear and nonlocal boundary condition for each ion concentration $c_k$, as 
$\Phi $ is obtained by solving an elliptic equation.

\medskip
We should also like to point out that the  presence of the projection $P(d)$ in the evolution equation for the director $d$
ensures that the constraint $|d|=1$  is preserved, provided it is imposed initially.

\medskip\noindent
In summary, the coupled system of equations \eqref{main-PDE} consists of

\begin{itemize}
    \item the {Nernst-Planck equations} for the flux of the ion concentrations  $c_k$,
    \item the {Poisson equation} for the electrostatic potential $\Phi$,
    \item the {incompressible Navier-Stokes equations} for the fluid solvent $v$, and 
    \item the {Eriksen-Leslie equations} for the director $d$.
\end{itemize}
\goodbreak
\bigskip\noindent
We employ the usual compatibility  conditions in the Ericksen-Leslie theory of elastic, isotropic, nematic liquid crystals, namely
\begin{equation}
\label{gamma-alpha}
\begin{aligned}
\gamma_1 &= \alpha_3 - \alpha_2,\qquad \\
\gamma_2 &= \alpha_6-\alpha_5.
\end{aligned}
\end{equation}
Moreover, we set
\begin{equation}
\label{beta}
\beta := \alpha_5 +\alpha_6.
\end{equation}

\medskip\noindent
For the remainder of the manuscript, we will impose the following conditions on the Leslie coefficients
\begin{equation}
\label{positivity}
\gamma_1>0, \quad \alpha_4>0, \quad \alpha_1\ge 0,  \quad 
4\beta\gamma_1  - (\gamma_2 +\alpha_2 +\alpha_3)^2 >0.
\end{equation}
\begin{remark}
\label{rem: Parodi}
{\rm (a)} 
The conditions in \eqref{positivity} are needed both to establish the local well-posedness of \eqref{main-PDE} and to prove the dissipativity of the energy in
 Proposition~\ref{pro: Lyapunov-equilibria}.
See also Remark~\ref{rem: Leslie}.

\smallskip\noindent
{\rm (b)}  
The Parodi relation states that
$\gamma_2 = \alpha_6-\alpha_5= \alpha_2+\alpha_3. $
In this case, conditions \eqref{positivity} simplify to 
\begin{equation*}
\gamma_1>0, \quad \alpha_4>0, \quad \alpha_1 \ge 0, \quad 
\beta\gamma_1  - \gamma_2 ^2 > 0.
\end{equation*}
We emphasize that our results do not require the Parodi relation to hold. 
\end{remark}

\medskip\noindent
There exists a large literature on various subsystems of the coupled system \eqref{main-PDE}. 
Perhaps most notably, the Navier-Stokes equations are the subject of very active mathematical research, both past and present. 
We refer the reader to \cite{BeVi22} for a modern survey of the field. Also see \cite{CoFo88, Te01}. 
For the Nernst-Planck equations, coupled to the Poisson and Navier-Stokes equations, global existence of weak and strong solutions and long-time behavior/stability have been established in both two and three dimensions, under a wide range of assumptions, by many authors. 
We refer the reader to the works \cite{AbIg21, BFS14, CIL22, CIL22b, CoIg19, FiSa17, Lee23, Lee24, Lee25, Sch09} and references therein. 
We note that the breadth of the literature for the Nernst-Planck-Poisson-Navier-Stokes equations is partially due to the wide range of choices one can make on the boundary conditions and other parameters (e.g., number of ionic species, spatial dimension). 
In contrast to our current work, however, the vast majority of this literature focuses on the special case of isotropic diffusion matrices and constant, isotropic dielectric permittivity matrices.

\medskip\noindent
Likewise,  there is a substantial body of work  for the Ericksen-Leslie equations for nematic liquid crystals, starting with the seminal work of 
Ericksen~\cite{Eri62} and Leslie~\cite{Les66, Les68} in the 1960s.
See also the comprehensive exposition in the monographs \cite{SoVi12, Vir94}.

The Ericksen–Leslie system provides a continuum description of the dynamics of nematic liquid crystals, coupling fluid velocity with the orientation of rod-like molecules. Since its formulation in the seminal works of Ericksen and Leslie, the system has attracted considerable attention in mathematical analysis, particularly with regard to the existence, uniqueness, and regularity of weak and strong solutions. Due to its complexity - arising from nonlinear coupling, geometric constraints, and anisotropic stresses - much of the early progress focused on simplified models that retain essential features while being more tractable.
One of the most influential simplifications is the so-called simplified Ericksen–Leslie system,  obtained by neglecting certain Leslie stress terms and assuming equal elastic constants. In this setting, the director field satisfies a transported harmonic map heat flow, coupled with the Navier–Stokes equations. 
For the full Ericksen–Leslie system, the analysis is considerably more involved due to the presence of additional stress terms.

Without going into further detail, we mention the contributions \cite{CRW13, DHW22, FFRS12, HLW25, HNPS16, HiPr17, HiPr18,HiPr19, HLW14, Lin91,LinLiu95, LinLiu00, LinWa08, LinWa14, MGL14, WXL13, WZZ13}, which address key aspects of existence and regularity of solutions, 
also in the context of the  Ginzburg-Landau approximation.  This list is far from being complete.

\medskip\noindent
The system~\eqref{main-PDE} introduces additional complexity and challenges beyond those already present in the subsystems described above. For instance, the coefficients of the elliptic Poisson equation now also depend on the director field. This, in turn, implies that the boundary conditions for the concentrations $c_k$
 depend in a nonlinear and nonlocal manner on both the concentrations themselves and the director $d$. 
 
 \medskip\noindent
System~\eqref{main-PDE} was derived in~\cite{CGLW16}, based on a continuum model for electrolytes in nematic liquid crystals that combines fluid dynamics, electrostatics, and director-field evolution. In that work, the evolution equation for the director $d$ is second order in time, see in particular equations (2.51)-(2.55),  whereas we consider here the 
Ericksen–Leslie equations for the director $d$ (with the additional  coupling term  $\ve_a P(d) (\nabla \Phi\otimes \nabla\Phi)d$)).

\smallskip\noindent
One way of approaching system \eqref{main-PDE} is from an energetic perspective, which asserts existence of an energy functional that is dissipative in time 
(see~Proposition~\ref{energy-est}). Weak solutions, based on uniform $L_2$-estimates from the energy, are natural candidates for a mathematical analysis. 
However, due to the nonlinear unit length constraint on 
the director,  global existence of weak solutions remains open even for the simplified Ericksen--Leslie system in three dimensions; see for instance the survey article \cite{LinWa14}. 
To circumvent this difficulty, one may relax the unit-length constraint and instead introduce a penalization potential that energetically favors $|d|=1$, leading to the so-called Ginzburg-Landau approximation. Within this framework, weak solutions can be constructed for the Ericksen--Leslie system; see, for example, \cite{LinLiu95, WXL13}. 

\noindent
Alternatively, one may employ a singular bulk potential that enforces a maximum principle of the form $|d|\le 1$, as in  \cite{FRSZ21}, where global-in-time weak solutions are obtained via weak sequential compactness arguments. 
In more detail, the authors in~\cite{FRSZ21} consider the system \eqref{main-PDE} in the setting of a Ginzburg-Landau potential of singular type for the director $d$,
 on the flat torus $\mathbb T^3$.  They analyze   the two-species case and prove a priori estimates that lead to weak sequential stability results.

\smallskip\noindent
In ongoing work \cite{DLS25}, we consider  \eqref{main-PDE} in two space dimensions and prove global existence of weak solutions. 
The approach relies on passing to the limit in a Ginzburg-Landau approximation
with a singular potential. Key steps include deriving uniform energy estimates and establishing the strong convergence of the director field utilizing a Pohozaev-type identity to handle the lack of compactness for the 
Ericksen stress tensor.

\smallskip\noindent
In this paper, we employ the theory of maximal regularity for quasilinear parabolic equations,
see for instance~\cite{PrSi16}, 
 to establish local well-posedness of strong solutions to system \eqref{main-PDE} in spatial dimensions
 $n \ge 2$.  The main result concerning existence  of solutions is contained in Theorem~\ref{thm: main}.
We show that strong  solutions $z=(\bc,v,d)$ with initial values $z^0=(\bc^0, v^0, d^0$) exist on a maximal time interval $[0, T_+(z^0))$, provided the Leslie coefficients 
 satisfy~\eqref{positivity}, and we show in addition that
 \begin{itemize}
\item solutions regularize and are smooth (and even analytic) in time on $(0, T_+(z^0))$.
 \vspace{2mm}
\item
 \eqref{main-PDE} generates a local Lipschitz continuous semiflow on a manifold $\cM_\mu$ that incorporates the nonlinear boundary conditions.
 \vspace{2mm}
 \item
 If $ | d^0|=1$, then $| d(t)|=1$ for each $t\in [0, T_+(z^0))$.
 \vspace{2mm}
 \item
 If $c^0_k > 0$ in $\bar\Omega$, then $c_k(t)>0 $ in $\bar\Omega$ as well for each $t\in [0,T_+(z^0))$. 
 \end{itemize} 
 
 \smallskip\noindent
 More precise statements can be found in Theorem~\ref{thm: main}.
  In Proposition~\ref{energy-est}, we establish an important energy estimate for the energy functional $E(t)$  associated with \eqref{main-PDE} and defined in~\eqref{E}.
   In Proposition~\ref{pro: Lyapunov-equilibria}, we show that the energy functional $E$  has a dissipative structure, 
  provided the Leslie coefficients satisfy \eqref{positivity}. This implies that  $-E$ constitutes a strict Lyapunov functional.
In addition, we  provide a complete  characterization of the equilibria of \eqref{main-PDE}.
In Theorem~\ref{thm:global existence}, we establish criteria for global existence of solutions.
In Appendix~\ref{Appendix A} we state and prove some technical results that are essential for our approach.
For instance, in Proposition~\ref{pro: Phi-time} we establish time regularity of solutions of a time-dependent elliptic problem and study the dependence of  these solutions on $(\bc, d)$.
These results are also of independent interest.
Finally, in Appendix~\ref{Appendix B} we provide a proof for the energy estimate stated in Proposition~\ref{energy-est}.
 
\medskip\noindent
 A first and essential step in the study of \eqref{main-PDE} is the reformulation of the coupled system  
 as a quasilinear parabolic evolution equation for the dynamic variables $(\bc,v,d)$. \\
The analysis of the resulting quasilinear evolution equation is far from straightforward,  
since the equations are strongly coupled and include an elliptic equation for the electrostatic potential $\Phi$, 
 whose coefficients depend on $(\bc,d)$.
 Additional difficulties arise as the velocity equation $\eqref{main-PDE}_5$ contains, through the term
$$
{\rm div}\,(\alpha_2 \od \otimes d + \alpha_3 d \otimes \od),
$$
first-order spatial derivatives of $\od$. Via the equation for $\od$ in $\eqref{main-PDE}_8$, these terms give rise to third-order spatial derivatives of the director $d$, an unusual feature.  
Moreover,  the expression ${\rm div}\,\sigma_L $ appearing in $\eqref{main-PDE}_5$ contains additional second order terms in $v$ besides $\alpha_4 \Delta v$
(which usually constitutes the dissipative term in the Navier--Stokes equations).

\medskip\noindent
All of this leads to significant challenges in the mathematical analysis of the system. In our approach, results obtained in ~\cite{HiPr17} turned out to be particularly important and useful for the analysis of the Ericksen-Leslie subsystem in \eqref{main-PDE}.
In order to obtain solutions via a fixed point argument, and to establish additional time regularity and the semiflow property,  we show that the full linearization introduced  in \eqref{linearized} gives rise to a  linear isomorphism between carefully chosen function spaces. This means that the non-autonomous system  \eqref{linearized} has the property of maximal regularity. 
The proof of this fact is technically demanding  and is carried out in Section~\ref{sec: linearized} and in Appendix~\ref{Appendix A}.

\bigskip\noindent
\textbf{Notation:} 
For the readers' convenience, we list here some notation and conventions used throughout the manuscript.

\medskip\noindent
For two vectors $a,b\in\R^n$, the Euclidean inner product is denoted by $a\cdot b$.
Given two matrices $A,B\in {\mathbb M}_n$, the Frobenius matrix inner product $A:B$ is given by
$
A:B={\rm trace} (AB^{\sT}),
$
where $B^{\sT}$ is the transpose of $B$.
Hence, with $A=[a_{ij}] $ and $B=[b_{ij}]$, $$A:B= \sum_{i,\, j} a_{ij} b_{ij}.$$
Suppose $\Omega$ is an open subset of $\R^n$.
If $u\in C^1(\Omega;\R^n)$, we set $\nabla u(x)=\partial_ju(x)\otimes e_j$
for $x\in\Omega$.
Hence, for $u=(u_1,\cdots , u_n)\in C^1(\Omega; \R^n)$, we have 
$$[\nabla u(x)]_{ij}= \partial_j u_i(x), \;\; 1\le i, j\le n, \;\; x\in \Omega.$$
We note that $\nabla u(x)$ corresponds to the Fr\'echet derivative of $u$ at $x\in\Omega$. 

If $A\in C^1(\Omega;\MM^n)$, its divergence ${\rm div}A$ is the vector function defined by
\begin{equation}
\label{divergence-matrix}
({\rm div }A)(x)=\partial_j (A(x)e_j), \;\; x\in \Omega.
\end{equation}
Hence, if $A=[a_{ik}]\in C^1(\Omega;\MM^n)$, its divergence is given by
 $${\rm  div} A(x) =\partial_j a_{ij}(x)e_i,\quad x\in\Omega.  $$
Thus, ${\rm div}A$ is defined by taking the divergence along rows of $A$. 
We note   that 
 \eqref{divergence-matrix} implies
\begin{equation}
\label{divergence-property}
({\rm div }A)\cdot u = {\rm div} (A^{\sf T}u)- A:\nabla u, \quad A\in C^1(\Omega;\MM^n), \ u\in C^1(\Omega;\R^n).
\end{equation}

For a matrix $A\in C^1(\Omega;\MM^n)$, we set 
$|\nabla A|^2=\partial_j A: \partial_jA.$

\smallskip
For functions 
$f,g\in L_2(\Omega; \R^m)$, 
$$(f|g)_\Omega =\int_\Omega f\cdot g\,dx$$ denotes the $L_2$-inner product.
For any  Banach space $X$, $s\ge 0$, $p \in (1,\infty)$,  
$W^s_p(\Omega;X)$ denote the  $X$-valued  Sobolev(-Slobodeckij)   spaces.
When the choice of $X$ is clear from the context, we will just write $W^s_p(\Omega)$. 
 In particular, we always write $W^s_p(\Omega)$ for $W^s_p(\Omega;\R)$. 

\goodbreak
\medskip
Given any $T\in (0,\infty]$, we will denote the interval $(0,T)$ by 
$J_T$.
For   $p\in (1,\infty)$ and $\mu\in (0,1]$, the $X$-valued $L_p$-spaces with temporal weight are defined by
$$
L_{p,\mu}(J_T;X):=\left\{ f: (0,T)\to X: \, [t \mapsto t^{1-\mu}f(t)] \in L_p(J_T;X)  \right\}.
$$
Similarly, for $k\in \N$,
$$
H^k_{p,\mu}(J_T;X):=\left\{ f \in    H^k_{1,loc}(J_T;X):\,  \pa_t^j f\in L_{p,\mu}(J_T;X), \, j=0,1,\ldots,k \right\}.
$$
For $s\in (0,1)$, the Sobolev-Slobodeckij spaces   with temporal  weights  are defined as
$$
W^s_{p,\mu}(J_T;X):= \{ u \in L_{p,\mu}(J_T;X): \, 
\|u\|_{W^s_{p,\mu}(J_T;X)} <\infty \} ,
$$
where  
\begin{equation}
\label{seminorm}
\begin{aligned}
& \|u\|_{W^s_{p,\mu}(J_T;X)} =  \|u\|_{L_{p,\mu}(J_T;X)}  + [u]_{W^s_{p,\mu}(J_T;X)}, \\
&[u]_{W^s_{p,\mu}(J_T;X)} := \left(\int_0^T  \int_0^t \tau^{p(1-\mu)} \frac{\| u(t)- u(\tau)\|_X^p}{(t-\tau)^{s p+1}}\, d\tau  d t \right)^{1/p},
\end{aligned}
\end{equation}
see \cite[Formula~(2.6)]{MeSc12}. 
$\|\cdot\|_{W^s_{p,\mu}(J_T;X)}$ is termed the intrinsic norm of  $W^s_{p,\mu}(J_T;X)$.
It holds that 
\begin{equation}
\label{Ws-BUC}
W^{s}_{p,\mu}(J_T; X)\hookrightarrow BU\!C(\bar J_T ;X),\quad s> 1-\mu +1/p,
\end{equation}
 see for instance \cite[Proposition 2.10]{MeSc12}.

For any two Banach spaces $X$ and $Y$, the notation $\cL(X,Y)$ stands for the set of all bounded linear operators from $X$ to $Y$ and $\cL(X):=\cL(X,X)$. 
${\rm Isom}\,(X,Y)$ denotes the subset of $\cL(X,Y)$ consisting of linear isomorphisms from $X$ to $Y$.

\section{The modified system}\label{sec: modified}
\noindent
We will introduce a modified version of system \eqref{main-PDE}, where we drop  the constraint $|d|=1$ and replace $\mu_k $ by $ (\ln c_k + z_k \Phi)$.
Before doing so, we first remark that in case $|d|=1$, 
\begin{equation}
\label{POmega}
P(d)\partial_t d = \partial_t d ,\quad  P(d) (v\cdot \nabla d )= v\cdot \nabla d  ,\quad  P(d)(\Omega(v) d) = \Omega(v)d , \quad   P(d)\od =\od,
\end{equation}
where $\od$ and  $ \Omega(v)$ are introduced in~\eqref{circ-d}.
In fact, the assertion $P(d)(\Omega(v)d)= \Omega(v)d$ is always satisfied.
Consequently,  in case $|d|=1$,  we can rewrite the Leslie stress tensor~\eqref{Leslie-stess} in the following equivalent form:
\begin{equation}
   \label{Leslie-stress-2}
   \begin{aligned}
  	\sigma_L=\sigma^m_L= \sigma^m_L(v,d) &= \mu_0 (D(v)d \cdot d)d\otimes d+\alpha_2 P(d)\od\otimes d+\alpha_3 d\otimes P(d)\od+\alpha_4 D(v) \\
	&\quad +\alpha_5 (P(d)D(v)d)\otimes d+\alpha_6 d\otimes P(d) D(v)d,  \\
	 \mu_0 &= \alpha_1 + \alpha_5+\alpha_6.
 \end{aligned}
 \end{equation}
 Employing  \eqref{BCk} and \eqref{equation-phi}, the modified system then reads as 
\begin{equation}
  \label{main-PDE-2}
	\left\{
  \begin{aligned}
    \pa_t c_k+  v \cdot \nabla c_k &= {\rm div}(\cD_k \nabla c_k)+ {\rm div}(z_k c_k \cD_k \nabla \Phi) &&\text{in} && \Omega, \\ 
    B_k(\bc,d) c_k  &=0  &&\text{on} && \partial \Omega, \\
    \pa_t v +  v\cdot \nabla  v  +\nabla \pi&=- {\rm div}(\nabla d\odot \nabla d) +{\rm div}\, \sigma^m_L 
    + {\rm div} \left((\nabla \Phi\otimes \nabla \Phi) \ve(d) \right)  \!\! &&\text{in} && \Omega, \\
    {\rm div}\, v &=0  &&\text{in} && \Omega, \\
     v &= 0 &&\text{on} && \partial \Omega, \\      
    \gamma_1\od+\gamma_2 P(d)D(v)d&= \Delta d+|\nabla d|^2 d   +\ve_a P(d)(\nabla \Phi\otimes \nabla\Phi)d &&\text{in} && \Omega,\\
    \partial_\nu d &=0 &&\text{on} && \partial \Omega, \\
  (c_k(0), v(0), d(0)) & =  (c^0_k , v^0, d^0) &&\text{in} && \Omega.\\
 \end{aligned}
  \right.
\end{equation}
\begin{remark}
\label{rem: equivalent}
It will be shown in Theorem~\ref{thm: main} that the constraint $ |d(t)| = 1 $ is preserved, provided it holds for the initial datum $ d^0 $.
Moreover, it will be shown that $ c_k(t) $ remains strictly positive if this property holds for the initial value $ c_k^0$.
Once these assertions are established, problems \eqref{main-PDE} and \eqref{main-PDE-2} are seen to be equivalent. 
\end{remark}

\noindent
Setting 
$ z=(\bc, v, d)\quad \text{with} \quad \bc= (c_1, \ldots, c_N)$
we rewrite \eqref{main-PDE-2} in equivalent form

\begin{equation}
  \label{main-PDE-3}
	\left\{
  \begin{aligned}
    \pa_t c_k  -  {\rm div}(\cD_k \nabla c_k)&= F_k (z) &&\text{in} && \Omega, \\ 
    B_k(z) c_k  &=0  &&\text{on} && \partial \Omega, \\
    \pa_t v  -  A_v(d)v + R_1(d) \partial_t d +\nabla \pi&= F_v(z)  &&\text{in} && \Omega, \\
    {\rm div}\, v &=0  &&\text{in} && \Omega, \\
     v &= 0 &&\text{on} && \partial \Omega, \\      
    \gamma_1\partial_t d +R_0(d)v - \Delta d&= F_d(z) &&\text{in} && \Omega,\\
    \partial_\nu d &=0 &&\text{on} && \partial \Omega, \\
  (c_k(0), v(0), d(0)) & =  (c^0_k , v^0, d^0) &&\text{in} && \Omega,\\
 \end{aligned}
  \right.
\end{equation}
where we collected the highest order terms on the left side, while the right side contains all the remaining lower order terms.
Using  \eqref{divergence-matrix}, \eqref{Leslie-stress-2} and assuming that $d\in C(\Omega; \R^n)$, one verifies that 
for $v\in C^2(\Omega; \R^n)$ with ${\rm div}\,v=0$,
\begin{equation}
\label{Av} 
A_v(d)v =  \alpha_4 \Delta v +  \sum_{l=0}^4 A_l (d) v,
\end{equation} 
with 
\begin{equation}
\label{Al-description}
\begin{aligned}
A_0(d) v &= \mu_0 ((\partial_j \nabla v) d \cdot d) (d \otimes d)e_j,  \\[2pt] 
A_1(d) v &=  \left(\frac{\alpha_5  + \alpha_2}{2}\right) \big( (P(d) (\partial_j  (\nabla v)^{\sf T}) d )\otimes d\big)e_j,  \\[2pt] 
A_2(d) v &=  \left(\frac{\alpha_5 - \alpha_2}{2}\right)\big( (P(d) ( \partial_j  \nabla v)d ) \otimes d\big)e_j,   \\[2pt] 
A_3(d) v &= \left(\frac{\alpha_6  +  \alpha_3}{2}\right)\big(d \otimes P(d) (\partial_j  (\nabla v)^{\sf T})d\big)\e_j,  \\[2pt] 
 A_4(d) v &=  \left(\frac{\alpha_6 - \alpha_2}{2}\right) \big(d\otimes P(d) (\partial_j  \nabla v)d \big)e_j.
\end{aligned}
\end{equation}
It follows from \eqref{POmega} that
\begin{equation}
\label{R0}
\begin{aligned}
R_0(d)v 
& = -\gamma_1 \Omega(v) d +\gamma_2 P(d) D(v) d = -\gamma_1 P(d) \Omega(v) d +\gamma_2 P(d) D(v)d \\
& = \left(\frac{\gamma_2-\gamma_1}{2}\right) P(d) (\nabla v )d + \left(\frac{\gamma_2+\gamma_1}{2}\right) P(d)( \nabla v)^{\sf T} d,  \\
\end{aligned}
\end{equation}
and moreover,
\begin{equation}
\label{R1}
\begin{aligned}
R_1(d) \partial_t d
& =   -\big(\alpha_2  P(d)\partial_j \partial_t d \otimes d+ \alpha_3  d\otimes  P(d) \partial_j\partial_t d \big)e_j . \\
\end{aligned}
\end{equation}
For the terms on the right hand side of \eqref{main-PDE-3}, we have
\begin{equation}
\begin{aligned}
\label{Fc-Fv-Fd}
F_k(z) & = {\rm div}\, (z_k c_k \cD_k \nabla \Phi (z)) - v\cdot \nabla c_k, \\
F_v(z) & = - {\rm div}\,( \nabla d \odot \nabla  d)   + {\rm div}\, \sigma^m_L(d, v)  - A_v(d)v  + R_1(d)\partial_t d  \\
           & \quad  + {\rm div} \left((\nabla \Phi\otimes \nabla \Phi)  \ve(d) \right)  - v\cdot \nabla v ,\\
F_d(z) & = |\nabla d|^2 d  +  \ve_a P(d)(\nabla \Phi (z)\otimes \nabla \Phi(z))d  -  \gamma_1  (v\cdot \nabla d).
\end{aligned}
\end{equation}
We should  like to point out  that for $d\in C^1(\Omega; \R^n)$,
$$R_1(d)=R_1(d, \nabla)\quad\text{and} \quad R_0(d)=R_0(d,\nabla)$$
 both act as first order differential operators in space.
Hence, the term $R_1(d)\partial_t d$ in \eqref{main-PDE-3} implies that the evolution equation for  $v$ actually 
contains {\em third-order spatial derivatives of~$d$},
as
 \begin{equation}
 \label{substitution}
 \gamma_1 R_1(d)\partial_t d = R_1(d) [\Delta d -R_0(d) v + F_d(z)],
 \end{equation}
and $\Delta d$ already involves second order derivatives.
 Moreover, $R_1(d) R_2(d)v$ acts  as second order operator on $v$.
Hence the usual term $-\mu\Delta v$ in the Navier-Stokes equations is here replaced by
$$ 
-\big(\alpha_ 4 \Delta + \sum_{l=0}^4 A_l(d) +\gamma^{-1}_1 R_1(d)R_0(d)\big) v,
$$
see also \eqref{linear-abstract-tilde}-\eqref{A(z)} below.
This all adds significant challenges to a mathematical analysis of system~\eqref{main-PDE-3}.
 
 \medskip
 In order to deal with the nonlinearities $(F_k(z), F_v(z), F_d(z))$ and the boundary operators $B_k(z)$ in \eqref{main-PDE-3},
we will in the sequel frequently assume that
\begin{equation}
\label{mu-condition}
n+2<p \quad \text{and} \quad \frac{1}{2}+\frac{n+2}{2p} <\mu \le 1.
\end{equation}
In particular, this guarantees the validity of the embedding
\begin{equation*}
W^{2\mu-2/p +j}_p(\Omega)   \hookrightarrow  C^{1+j}(\overline{\Omega} ),\qquad j=0,1.
\end{equation*}

\section{The linearized problem and maximal regularity}\label{sec: linearized}
\noindent
We now consider the following autonomous linear version of  \eqref{main-PDE-3}
\begin{equation}
  \label{linear-MR}
	\left\{
  \begin{aligned}
    \pa_t c_k  -  {\rm div}(\cD_k \nabla c_k)&= f_k &&\text{in}  &&\Omega, \\ 
    B_k( \wt \bc, \wt d) c_k  &=g_k  &&\text{on} &&\partial \Omega, \\
    \pa_t v  -  A_v(\wt d)v + R_1(\wt d) \partial_t d +\nabla \pi&= f_v &&\text{in} &&\Omega, \\
    {\rm div}\, v &=0  &&\text{in} && \Omega, \\
     v &= 0 &&\text{on} &&\partial \Omega, \\      
    \gamma_1\partial_t d +R_0(\wt  d)v - \Delta d&= f_d  &&\text{in} &&\Omega,\\
    \partial_\nu d &=0 &&\text{on} &&\partial \Omega, \\
  (c_k(0), v(0), d(0)) & =  (c^0_k , v^0, d^0) &&\text{in} && \Omega,
 \end{aligned}
  \right.
\end{equation}
where $(\wt \bc, \wt d)\in L_p(\Omega; \R^N)\times C^1(\bar\Omega; \R^n)$ and $k=1,\ldots, N$. We will establish results on optimal regularity of solutions; 
that is, we establish {\em maximal regularity results} for the linear system~\eqref{linear-MR}.

Before stating the result,
 it will be instructive to motivate our choice of function spaces.
For the ion concentrations $c_k$, we require 
$$c_k\in H^1_{p,\mu}(J_T; L_p(\Omega))\cap L_{p,\mu}(J_T; H^2_p(\Omega)),$$
with $p\in (1,\infty)$ and $\mu\in (1/p, 1]$, which then implies $f_k\in L_{p,\mu}(J_T; L_p(\Omega))$.
The regularity of $c_k$ yields
 \begin{equation}
 \label{F-mu-1}
 g_k \in  \FF^1_\mu(J_T):=W^{1/2 -1/2p}_{p,\mu}(J_T; L_p(\partial\Omega; \R ))\cap L_{p,\mu}(J_T; W^{1-1/p}_p(\partial\Omega)),
 \end{equation}
 see 
 \cite[Theorem 3.4, Theorem 4.5]{MeSc12}, or  \cite[Theorem 6.3.2]{PrSi16} .
Next we require that
$$v\in H^1_{p,\mu}(J_T; L_p(\Omega; \R^n))\cap L_{p,\mu}(J_T; H^2_p(\Omega; \R^n)), \quad 
\nabla \pi  \in L_{p,\mu}(J_T; L_p(\Omega; \R^n)).$$
This implies 
 $f_v\in L_{p,\mu}(J_T; L_p(\Omega;\R^n))$.
  Due to the presence of the term $R_1(\wt d)\partial_t d $ in the equation for $v$, we finally require
 $$d\in H^1_{p,\mu}(J_T; H^1_{p}(\Omega; \R^n))\cap L_{p,\mu}(J_T; H^3_p(\Omega; \R^n)),$$
and then $f_d\in L_{p,\mu}(J_T; H^1_p(\Omega; \R^n))$.
The embedding 
$$H^1_{p,\mu}(J_T; H^j_p(\Omega; \R^\ell)) \cap L_{p,\mu}(J_T; H^{2+j}_p(\Omega; \R^\ell))
 \hookrightarrow BU\!C(\bar J_T; W^{2(\mu- 1/p)+j}_p(\Omega; \R^\ell)), \quad j=0,1,$$
 see for instance \cite[Proposition 3.1]{PrSi04}, implies 
 $$c^0_k\in W^{2(\mu- 1/p)}_p(\Omega ),\quad v^0\in W^{2(\mu- 1/p)}_p(\Omega ;\R^n), \quad
 d^0\in W^{2(\mu- 1/p)+1}_p(\Omega ;\R^n).
 $$
Within the above regularity setting, we need to show that the boundary operators $B_k(\wt \bc, \wt d)$ are properly defined.
To this end, we establish the following result. 
\begin{lemma}
\label{elliptic-tilde}
Suppose that \eqref{permittivity}.
Then the elliptic boundary value problem 
    \begin{equation*}
    \label{equation-phi-2}
    \left\{
    \begin{aligned}
    -{\rm div}\, (\ve(\wt d)\nabla \Phi)&=\sum_{k=1}^N z_k \wt c_k &&\text{in} &&\Omega,\\
           \Phi &= 0 &&\text{on} &&\partial \Omega, \\
    \end{aligned}
    \right.
    \end{equation*}
has for each 
$$(\wt \bc,\wt d)\in L_p(\Omega; \R^N)\times C^1(\bar\Omega; \R^n),\quad \wt\bc=(\wt c_1,\ldots , \wt c_N),$$
a unique solution $\Phi= \Phi(\wt \bc, \wt d)\in H^{2}_{p,D}(\Omega)$, where
$H^{2}_{p,D}(\Omega):=\{u\in H^{2}_p(\Omega): u=0 \text{ on } \partial\Omega\}$.
The mappings 
\begin{equation*}
\begin{aligned}
&[(\wt \bc, \wt d) \mapsto \Phi (\wt \bc, \wt d)] :  L_p(\Omega; \R^N)\times C^1(\bar\Omega; \R^n) \to H^2_p(\Omega), \\
&[(\wt \bc, \wt d) \mapsto B_k(\wt \bc, \wt d)]:   
L_p(\Omega; \R^N)\times C^1(\bar\Omega; \R^n)\to \cL(\EE^1_{1,\mu}(J_T),\, \FF^1_\mu(J_T))
\end{aligned} 
\end{equation*}
are  real analytic,  where $\EE^1_{1,\mu}(J_T)= H^1_{p,\mu}(J_T; L_p(\Omega))\cap L_{p,\mu}(J_T; H^2_p(\Omega))$
and $\FF^1_{\mu}(J_T)$ is defined in \eqref{F-mu-1}. 
\end{lemma}
\begin{proof}
By \eqref{permittivity}, 
$$(\ve (\wt d)\xi )\cdot \xi = \ve_{\perp} |\xi|^2 + \ve_a (d\cdot \xi)^2 \ge \ve_{\perp} | \xi |^2,\quad \xi\in \R^n. $$
Hence, the differential operator $-{\rm div}\, (\ve(\wt d)\nabla \cdot)$ is strongly elliptic, and  the existence and uniqueness of a solution 
$\Phi =\Phi(\wt \bc, \wt d)\in H^2_{p,D}(\Omega)$ follows from standard elliptic theory, see for instance
\cite[Theorem 9.15 and Lemma 9.17]{GiTr01}.

\smallskip\noindent
For $\wt d\in C^1(\bar \Omega; \R^n)$ let $L(\wt d):= -{\rm div}\, (\ve(\wt d)\nabla \,\cdot \,)$.
Then, by unique solvability,
$$
L(\wt d)\in {\rm Isom}(H^2_{p,D}(\Omega), L_p(\Omega))\quad\text{and} \quad
L(\wt d)^{-1}\in \cL (L_p(\Omega), H^2_{p}(\Omega)).$$
It follows that
\begin{equation*}
\Phi(\wt \bc, \wt d) = L(\wt d)^{-1} \sum_{k=1}^N z_k \wt c_k \in H^2_{p,D}(\Omega).
\end{equation*}
As $[\wt d \mapsto \ve (\wt d)]: C^1(\bar \Omega;  \R^n) \to C^1(\bar\Omega ; \R^{n \times n})$ 
is real analytic, we  infer that
$$ [\wt d \mapsto L(\wt d)]: C^1(\bar\Omega; \R^n) \to \cL(H^2_{p,D}(\Omega), L_p(\Omega))$$
is real analytic as well. 
The first assertion now follows from the fact that inversion is real analytic. 
It follows from trace theory that
\begin{equation}
\label{Phi-boundary}
b_{k, \pa \Omega} :=
b_{k, \pa \Omega} (\wt \bc, \wt d):=
\nu\cdot {\rm tr}_{\partial\Omega} (z_k \cD_k \nabla \Phi(\wt \bc, \wt d))\in W^{1-1/p}_p(\partial\Omega).
\end{equation}
 For the last assertion, we note that
$$
[c_k \mapsto \nu \cdot {\rm tr}_{\partial\Omega} (\cD_k \nabla c_k)]\in 
 \cL(\EE^1_{1,\mu}(J_T),\, \FF^1_\mu(J_T)),
$$ 
see  \cite[Theorem 3.4, Theorem 4.5]{MeSc12} for instance, and
$$
[(b_{k,\pa\Omega}, {\rm tr}_{\pa\Omega} c_k) \mapsto b_{k,\pa\Omega} {\rm tr}_{\partial\Omega} c_k]: W^{1-1/p}_p(\partial\Omega) \times \FF^1_\mu(J_T) \to \FF^1_\mu(J_T)
$$
is well-defined, bilinear and continuous, as $\FF^1_{\mu}(J_T)$ is a multiplication algebra.  
\end{proof}

\goodbreak
\medskip\noindent
We can now establish the following maximal regularity result for the linear system \eqref{linear-MR}.
\begin{theorem}
\label{thm: linear-MR}
Suppose  that \eqref{positivity}  and \eqref{mu-condition} hold.
Let
$$(\wt \bc, \wt d)\in L_p(\Omega; \R^N) \times  C^1(\bar\Omega; \R^n)$$
be given.
Then for any $T>0$, the linear system  \eqref{linear-MR} admits a unique solution in the regularity class
\begin{equation*}
\begin{aligned}
c_k &\in H^1_{p,\mu}(J_T; L_p(\Omega))\cap L_{p,\mu}(J_T; H^2_p(\Omega)),  \\
v &\in H^1_{p,\mu}(J_T; L_p(\Omega; \R^n))\cap L_{p,\mu}(J_T; H^2_p(\Omega; \R^n)),  \\
d &\in  H^1_{p,\mu}(J_T; H^1_p(\Omega; \R^n))\cap L_{p,\mu}(J_T; H^3_p(\Omega; \R^n)),  \\
\nabla \pi  &\in L_{p,\mu}(J_T; L_p(\Omega; \R^n)), 
\end{aligned}
\end{equation*}
provided that
\begin{equation*}
\begin{aligned}
 &{\sf f}_c  = (f_1,\ldots, f_N) \in  L_{p,\mu}(J_T; L_p(\Omega; \R^N)),  \\
&{\sf g}     =(g_1,\ldots, g_N) \in W^{1/2-1/2p}_{p,\mu}(J_T; L_p(\partial \Omega ; \R^N))\cap L_{p,\mu}(J_T; W^{1-1/p}_p(\partial \Omega; \R^N)) ,  \\
      & f_v  \in L_{p,\mu}(J_T; L_p(\Omega; \R^n)), \\
      & f_d  \in L_{p,\mu}(J_T; H^1_p(\Omega; \R^n))
\end{aligned}
\end{equation*}
and the initial data satisfy the regularity and compatibility conditions 
\begin{equation}
\label{initial-compatibility}
\begin{aligned}
& c^0_k \in W^{2\mu -2/p}_p (\Omega),          && B_k(\wt z)c^0_k = g_k(0)  \text{ on }\ \partial\Omega, \\
& v^0 \in W^{2\mu -2/p}_p (\Omega; \R^n),          && {\rm div}\,v^0=0\text{ in } \ \Omega,     \quad v^0=0   \text{ on }\ {\partial\Omega},    \\
& d^0 \in W^{1+2\mu  -2/p}_p(\Omega; \R^n),   \quad   && \partial_\nu d^0 =0  \text{ on }\ {\partial\Omega}. \\
\end{aligned}
\end{equation}
\end{theorem}
\begin{proof}
We begin by focusing our attention on the subsystem
\begin{equation}
  \label{linear-vd}
	\left\{
  \begin{aligned}
    \pa_t v  -  A_v(\wt d)v + R_1(\wt d) \partial_t d +\nabla \pi&= f_v &&\text{in} &&\Omega, \\
    {\rm div}\, v &=0  &&\text{in} &&\Omega, \\
     v &= 0 &&\text{on} &&\partial \Omega, \\      
    \gamma_1\partial_t d +R_0(\wt  d)v - \Delta d&= f_d  &&\text{in} &&\Omega,\\
    \partial_\nu d &=0 &&\text{on} &&\partial \Omega, \\
  ( v(0), d(0)) & =  (v^0, d^0) &&\text{in} &&\Omega.\\
 \end{aligned}
  \right.
\end{equation}
Fortunately, we can rely here on results already established by Hieber and Pr\"uss in \cite{HiPr17}, ensuring existence and uniqueness of solutions
to  \eqref{linear-vd}. 

Setting  $\gamma =\mu_V$ and comparing the Leslie coefficients $\gamma_1, \gamma_2$ and 
 $\alpha_1,\ldots, \alpha_6$  with the coefficients 
$\mu_s, \mu_0, \mu_V, \mu_D, \mu_L, \mu_P$
in \cite[equations (1.4), (1.9)]{HiPr17},  we can infer  that
\begin{equation}
\label{comparison}
\begin{aligned}
 \gamma_1 & = \mu_V=\gamma, \quad \gamma_2 = -\mu_D , \\
 \mu_0  &= \alpha_1 + \alpha_5 +\alpha_6, \\
\alpha_4 &= \mu_s, \\
\alpha_2 &= - \left(\frac{\mu_D + \mu_V +2 \mu_P} {2\gamma}\right)\mu_V,  \\
\alpha_3 &= - \left(\frac{\mu_D - \mu_V +2 \mu_P} {2\gamma}\right)\mu_V,  \\
\alpha_5 & = \left(\frac{\mu_D + \mu_V +2 \mu_P} {2\gamma}\right)\mu_D + \frac{\mu^2_P}{2\gamma} + \frac{\mu_L}{2}, \\
\alpha_6 & =  \left(\frac{\mu_D - \mu_V +2 \mu_P} {2\gamma}\right)\mu_D + \frac{\mu^2_P}{2\gamma} + \frac{\mu_L}{2}. \\
\end{aligned}
\end{equation}
A short computation then yields
\begin{equation*}
\begin{aligned}
& \alpha_2 +\alpha_3 = -(\mu_D  +2\mu_P) =  \gamma_2 - 2\mu_P,  \\
& \alpha_5 + \alpha_6  
=\left(\frac{\mu_D + 2\mu_P }  {\mu_V } \right) \mu_D +\frac{\mu_P^2}{\mu_V} +\mu_L
= (\alpha _2 + \alpha_3) \frac{\gamma_2 } {\gamma_1 }  +\frac{\mu_P^2}{\gamma_1} +\mu_L. \\
\end{aligned}
\end{equation*}
Therefore, the coefficients $\mu_P$ and $\mu_L$ in \cite{HiPr17} can be expressed by the Leslie coefficients as 
\begin{equation}
\label{muP-muL}
\begin{aligned}
& 2\mu_P = \gamma_2 - (\alpha_2 +\alpha_3),  \\ 
& \mu_L = \frac{1}{\gamma_1}\big[ (\alpha_5 +\alpha_6)\gamma_1 - (\alpha_2 +\alpha_3)\gamma_2 -\mu_P^2 \big]
= \frac{1}{\gamma_1}\big[(\alpha_5 +\alpha_6)\gamma_1  -\frac{1}{4}(\gamma_2 +(\alpha_2 +\alpha_3))^2  \big]\\
& \quad = \frac{1}{\gamma_1}\big[\beta\gamma_1  -\frac{1}{4}(\gamma_2 +(\alpha_2 +\alpha_3))^2  \big].
\end{aligned} 
\end{equation}
Hence, \eqref{comparison}-\eqref{muP-muL} in conjunction with \eqref{beta}-\eqref{positivity} imply
 that the conditions $\mu_s>0, \gamma>0$ and $\mu_0,\mu_L\ge 0$ of  Assumption (P) in \cite{HiPr17} are satisfied.
 
 \medskip\noindent
 It follows from \cite[Theorem 5.1]{HiPr17} 
 (the temperature variable $\theta$ is omitted here, as we consider the isothermal case) that  system \eqref{linear-vd} 
admits  a unique solution 
 $(v,d,\pi)$ with the regularity properties stated in Theorem~\ref{thm: linear-MR} above.

In fact,  \cite[Theorem 5.1]{HiPr17} only establishes maximal regularity in function spaces without time weights and for zero initial conditions. 
However, an inspection
of the proof shows that  \cite[Theorem 5.1]{HiPr17} extends to our situation,  as the underlying solvability results are based on 
Sections 6.3 and  7.3  in \cite{PrSi16} which allow for time weights and non-trivial initial conditions that satisfy the regularity and compatibility conditions of
our theorem. 
We also refer to \cite[Theorem 5.1]{HiPr19} and \cite[Theorem 6.2]{HLW25}, where a related situation is considered.

\medskip\noindent
Finally, we observe that the remaining  equations
 \begin{equation}
  \label{linear-ck}
	\left\{
  \begin{aligned}
    \pa_t c_k  -  {\rm div}(\cD_k \nabla c_k)&= f_k &&\text{in} &&\Omega, \\ 
    B_k(\wt \bc, \wt d) c_k  &=g_k  &&\text{on} &&\partial \Omega, \\
     c_k(0) & =  c^0_k&&\text{in} && \Omega,\\
 \end{aligned}
  \right.
\end{equation}
in \eqref{linear-MR} are completely decoupled from the equations for $(v,d,\pi)$.
Condition \eqref{D-positive-definite} implies that 
$$(-{\rm div}(\cD_k \nabla \,\cdot\,), B_k(\wt\bc, \wt d)),\quad k=1,\ldots, N,$$ 
satisfies the assumptions of \cite[Theorem 6.3.2]{PrSi16}.
We can then infer that the linear parabolic problem \eqref{linear-ck} 
has a unique solution in the maximal regularity class asserted.

This completes the proof of Theorem~\ref{thm: linear-MR}.
\end{proof}

\begin{remark}
\label{rem: Leslie}
Assumption (P) in \cite{HiPr17} in conjunction with~\eqref{muP-muL} implies that Theorem~\ref{thm: linear-MR}
remains valid under the condition
\begin{equation*}
\gamma_1>0, \quad \alpha_4>0, \quad \alpha_1 + \beta \ge 0, \quad 
4\beta\gamma_1  - (\gamma_2 +\alpha_2 +\alpha_3)^2 \ge 0.
\end{equation*}
We imposed the slightly stronger condition~\eqref{positivity} in order to establish, in addition, that $-E$ is a strict Lyapunov functional
for \eqref{main-PDE}, see~Proposition~\ref{pro: Lyapunov-equilibria}. 

\smallskip\noindent
We also note that the Parodi relation $\gamma_2= \alpha_2+\alpha_3$ is satisfied exactly when $\mu_P =0$.
\end{remark}
Setting $\wt z= (\wt \bc, \wt v,\wt d)$ and $z=(\bc, v,d)$, 
we can rewrite the linear system \eqref{linear-MR} in a more condensed and abstract (equivalent) form.

Indeed, substituting the expression for $\partial_t d$, as given in the equation for the director $d$,
into $R_1(\wt d) \partial_t d$ and applying the Helmholtz projection, the linear system \eqref{linear-MR} can be cast  as
\begin{equation}
\label{linear-abstract-tilde}
\left\{
\begin{aligned}
 \partial_t z + \sA(\wt z) z &= \mathsf f(t) &&\text{in} \quad\Omega, \\
\sB(\wt z) z & = \mathsf g(t) &&\text{on}  \quad \partial\Omega, \\
       z(0)& = z^0 &&\text{in} \quad \Omega,
\end{aligned}
\right.
\end{equation}
where
\begin{equation}
\label{A(z)}
\begin{aligned}
\sA(\wt z): & =
\begin{bmatrix}
-A_c &0 &0  
\\[5pt] 
0 & -\PP_H \left(A_v(\wt d) +\frac{1}{\gamma_1} R_1(\wt d) R_0(\wt d) \right) & \frac{1}{\gamma_1} \PP_H R_1 (\wt d) \Delta  \\[5pt] 
0&  \frac{1}{\gamma_1} R_0(\wt d)                                                           & - \frac{1}{\gamma_1}  \Delta
\end{bmatrix},
\\[5pt]
\quad \sB(\wt z) z : &= B_c(\wt \bc, \wt d) \bc .
\end{aligned}
\end{equation}

\medskip\noindent
Here, $\sA_c$ and $B_c(\wt z)$ are given by
\begin{equation}
\label{Ac}
\begin{aligned}
&A_c: = -{\rm diag}\, 
[{\rm div}(\cD_1 \nabla \cdot\;), \ldots, {\rm div}( \cD_N \nabla \cdot\;) ], \\[5pt]
&B_c(\wt z): = {\rm diag}\,[B_1(\wt z),\ldots, B_N(\wt z)], \quad B_k(\wt z) u 
            = \nu \cdot \cD_k\nabla u +  (\nu \cdot  z_k\cD_k \nabla \Phi(\wt \bc, \wt d) )u. \\[5pt]
\end{aligned}
\end{equation}
We note once again that
$R_1(\wt d)R_0(\wt d)$, with $\wt d\in C^1(\Omega)$, acts as a second-order differential operator on $v$, being the composition of two first-order operators,
while $R_1(\wt d) \Delta $ acts as a third-order differential operator on $d$.

\medskip\noindent
We will also introduce a condensed notation for the function spaces involved in our analysis.
For this, let
\begin{equation}
\label{X0}
  X_0:=L_{p}(\Omega;\R^N)\times L_{p, \sigma}(\Omega;\R^n)\times H^1_p(\Omega;\R^n),\quad 1<p<\infty.
\end{equation}
Here,
$
L_{p,\sigma}(\Omega;\R^n) :=\PP_H (L_p(\Omega;\R^n) )
$
is the space of all solenoidal vector fields in $L_p(\Omega; \R^n)$,
with 
$$\PP_H: L_p(\Omega;\R^n) \to L_{p,\sigma}(\Omega;\R^n)$$ 
the Helmholtz projection.
For all $s\geq 0$, we define
\begin{align*}
W^s_{p,\sigma}(\Omega;\R^n) :& = W^s_p(\Omega;\R^n)\cap L_{p,\sigma}(\Omega;\R^n).
\end{align*}
Moreover, we set
\begin{equation}
\label{X1}
X_1=H^1_p(\Omega; \R^N)\times \{v\in H^1_{p,\sigma}(\Omega; \R^n): v=0 \text{ on } \partial\Omega\}
\times \{d\in H^1_{p}(\Omega; \R^n): \partial_\nu d=0 \text{ on } \partial\Omega\}
\end{equation}
 and we denote the space of initial data as
\begin{equation}
\label{X-gamma-mu}
  X_{\gamma,\mu}:= (X_0,X_1)_{\mu-1/p,p}
\end{equation}        
for $\mu\in (1/p,1]$.
For $\mu$ satisfying \eqref{mu-condition}, we infer from \cite[Theorem 3.4]{Ama00}  and \cite[Theorem~4.3.3]{Tri78} that   
\begin{equation}
\label{interpolation}
(\bc, v, d)\in X_{\gamma,\mu} \ \Leftrightarrow \
\left\{
\begin{aligned}
& \bc\in W^{2\mu-2/p}_{p}(\Omega;\R^N), && \\
& v\in W^{2\mu-2/p}_{p, \sigma}(\Omega; \R^n),  && {\rm div}\,v^0=0\text{ in } \ \Omega,     \quad v^0=0   \text{ on }\ {\partial\Omega},    \\
& d\in  W^{2\mu-2/p+1}_{p}(\Omega;\R^n) ,    && \partial_\nu d^0 =0  \text{ on }\ {\partial\Omega}. \\\end{aligned}
\right.
\end{equation}
\medskip
Additionally,  given any $T\in (0,\infty]$, we define
\begin{equation}
\label{EE_1}
\begin{aligned}
 & \EE_{1,\mu}(J_T):= H^1_{p,\mu}(J_T; X_0)   \cap L_p(J_T; X_1),\\
& \EE_{0,\mu}(J_T):= L_{p,\mu}(J_T; X_0), \\
&
\FF_\mu(J_T)       :=    W^{1/2-1/2p}_{p,\mu}(J_T;L_p(\partial\Omega;\R^N))\cap L_{p,\mu}(J_T; W^{1-1/p}_p (\partial\Omega;\R^N)),
\end{aligned}
\end{equation}
where $X_0$ and $X_1$ are defined in~\eqref{X0} and \eqref{X1}, respectively.
Finally, for future analysis, we also introduce the spaces with vanishing  trace   at $t=0$:
\begin{equation*}
\begin{aligned}
	&{_0}\EE_{1,\mu}(J_T):=\{z\in \EE_{1,\mu}(J_T):  z(0)=0\}, \\
	 &{_0}\FF_\mu(J_T)  :=\{ g\in \FF_\mu(J_T):  g(0)  = 0 \},
\end{aligned}
\end{equation*}

\medskip\noindent
We can now establish a maximal regularity result for  \eqref{linear-abstract-tilde} .

\begin{proposition}
\label{pro: abstract-linear-tilde}
Suppose  that \eqref{positivity}  and \eqref{mu-condition} hold.
Let
$$(\wt \bc, \wt d)\in L_p(\Omega; \R^N) \times C^1(\bar\Omega; \R^n)$$
be given.
Then for any $T>0$, the linear system  \eqref{linear-abstract-tilde} admits a unique solution $z\in \EE_{1,\mu}(J_T)$,
provided 
\begin{equation*}
(\mathsf f, \mathsf g, z^0)\in \DD_\mu((\wt \bc, \wt d), T):=\{\EE_{0,\mu}(J_T) \times \FF_\mu(J_T)\times X_{\gamma,\mu} : \sB(\wt z) z^0= {\sf g}(0)\}.
\end{equation*}
There exists a constant $M=M(T)>0$ such that
$$ \| z\|_{\EE_{1,\mu}(J_T)}\le M \| (f,g,z_0)\|_{\EE_{1,\mu}(J_T)\times \FF_\mu( J_T) \times X_{\gamma,\mu}},
\quad (\mathsf f, \mathsf g, z^0)\in \DD_\mu((\wt \bc, \wt d), T).
$$
Moreover, for each fixed $T_*>0$  and each compact subset $K$ of $L_p(\Omega;\R^N)\times C^1(\bar\Omega)$ there exists a constant
$M=M(T_*, K)>0$ such that 
 \begin{equation}
 \label{estimate-uniform}
 \| z\|_{{_0}\EE_{1,\mu}(J_T)}\le M \| (f,g)\|_{\EE_{1,\mu}(J_T)\times \FF_\mu( J_T)}, \quad (\wt \bc, \wt d)\in K,\quad T\in (0,T_*],
\end{equation}
provided $z^0=0$ and $\sf g(0)=0$.
\end{proposition}
\goodbreak
\goodbreak
\begin{proof}
The first part of the assertion follows  from Theorem~\ref{thm: linear-MR}. 
To this end,  we first observe that the compatibility conditions \eqref{initial-compatibility} are satisfied by \eqref{interpolation}
and  $\cB(\wt z)\bc^0=\mathsf g(0)$. By assumption, we have
$$
\mathsf f =(\mathsf f_1, \mathsf f_2, \mathsf f_3)
\in L_{p,\mu}(J_T; L_p(\Omega;\R^N))\times L_{p,\mu}(J_T; L_{p,\sigma}(\Omega; \R^n))\times L_{p,\mu}(J_T; H^1_p(\Omega; \R^n)).
$$
For 
$$\mathsf f_c := \mathsf f_1,\quad f_d := \gamma_1 \mathsf f_3,\quad   f_v:= \mathsf f_2  + \PP_H  R_1(\wt d) \mathsf f_3,$$
let $(\bc, v, d, \pi)$ be the unique solution of \eqref{linear-MR} asserted in Theorem~\ref{thm: linear-MR}. 
Applying the Helmholtz projection $\PP_H$ to the equation for $v$ in \eqref{linear-MR} and using 
that  $\PP_H \nabla \pi =0$, $\PP_H \mathsf f_2 = \mathsf f_2$ and $\PP_H v=v$ (as ${\rm div}\,v=0$),  one immediately verifies  that $v$ satisfies
\begin{equation}
\label{PH-v}
\partial_t v -\PP_H A_v(\wt d) v + \PP_H R_1(\wt d) \partial_t d = f_v = \mathsf f_2  + \PP_H  R_1(\wt d) \mathsf f_3.
\end{equation}
Substituting the expression
$\partial_t d = \frac{1}{\gamma_1} (f_d  -R_0(\tilde d) v +\Delta d)$
into \eqref{PH-v} shows that $(v,d)$ is a solution of
\begin{equation*}
\begin{aligned}
\partial_t v -\PP_H \big(A_v(\wt d) v  + (1/{\gamma_1}) R_1(\wt d) R_0(\tilde d)\big) v + (1/\gamma_1)\PP_H R_1(\wt d)\Delta d  &=  \mathsf f_2, \\
 \partial_t d +(1/\gamma_1) R_0(\wt  d)v - (1/\gamma_1) \Delta d&= \mathsf f_3.
\end{aligned}
\end{equation*}
In conclusion, $z=(\bc, v, d)$ belongs to $ \EE_{1,\mu}(J_T)$ and is the unique solution of  \eqref{linear-abstract-tilde}.
We have shown that the mapping
\begin{equation*}
\label{uniform-estimate}
\LL(\wt z) := (\partial_t + \cA(\wt z), \sB(\wt z), \gamma_0): \EE_{1,\mu}(J_T) \to  \DD_\mu((\wt \bc, \wt d), T).
\end{equation*}
is linear, bounded, and invertible. The open mapping theorem then implies existence of a constant $M=M(T)$ such that the first estimate
of the proposition holds. (By the mapping properties of ${\sB(\wt z)}$, one readily verifies that $ \DD_\mu((\wt \bc, \wt d), T)$ is a Banach space).

\medskip\noindent
It remains to show~\eqref{estimate-uniform}.
For this, we first observe that  the mapping
$$[(\wt \bc, \wt d) \mapsto (\sA(\wt d), \sB(\wt \bc, \wt d)) ] : L_p(\Omega; \R^N) \times C^1(\bar\Omega; \R^n) \to \cL(\EE_{1,\mu}(J_T), \EE_{0,\mu}\times\FF_\mu(J_T))$$
is continuous (in fact, real analytic).
It is not difficult to verify the assertion for $\sA(\wt d)$, while  the  assertion for $\cB(\wt \bc, \wt d)$ follows from  Lemma~\ref{elliptic-tilde}. 

Letting $\LL_0(\wt \bc, \wt d):= (\partial_t +\sA(\wt d), \sB(\wt \bc, \wt d))$,
we can conclude from the solvability results for   \eqref{linear-abstract-tilde} already established that
$$
[(\wt \bc, \wt d)\mapsto \LL_0(\wt \bc, \wt d)]: L_p(\Omega;\R^N)\times C^1(\bar\Omega; \R^n) \to {\rm Isom}({_0}\EE_{1,\mu}(J_T), \EE_{0,\mu}(J_T) \times {_0}\FF_\mu(J_T))
$$
is defined and continuous, as we are taking $z^0=0$ and $\mathsf g(0)=0$.  \\
 Here, $ {\rm Isom}({_0}\EE_{1,\mu}(J_T), \EE_{0,\mu}(J_T) \times {_0}\FF_\mu(J_T))$ has the topology of 
$\cL(\EE_{1,\mu}(J_T), \EE_{0,\mu}(J_T) \times \FF_\mu(J_T)).$
As inversion is continuous, we obtain that the mapping
$$
[(\wt \bc, \wt d)\to \LL_0(\wt \bc, \wt d)^{-1}]: L_p(\Omega; \R^N)\times C^1(\bar\Omega; \R^n) \to \cL( \EE_{0,\mu}(J_T) \times {_0}\FF_\mu(J_T), {_0}\EE_{1,\mu}(J_T)))
$$
is continuous as well. This implies \eqref{estimate-uniform} for each fixed $T\in (0, T_*)$. Independence of $T\in (0,T_*)$ can then be obtained 
by an extension-restriction argument as in~\cite[Proposition 4.1(b)]{DSS24}. 
\end{proof}
\noindent
The nonlinear system \eqref{main-PDE-3} will also be written in the more condensed (and equivalent) form
\begin{equation}
\label{nonlinear abstract}
\left\{
\begin{aligned}
 \partial_t z + \sA(z) z &= \sF(z) &&\text{in} \quad\Omega \\
\sB(z) z & = 0 &&\text{on}  \ \  \partial\Omega \\
       z(0)& = z^0 &&\text{in} \quad \Omega,
\end{aligned}
\right.
\end{equation}
where the operators $(\sf A(z), \sf B(z))$ are defined in \eqref{A(z)}-\eqref{Ac} and 
\begin{equation}
\label{F(z)}
\begin{aligned}
\sF(z): =
\begin{bmatrix}
 F_{\bc }(z)  \\[5pt] 
  \PP_H \big(F_v(z) -\frac{1}{\gamma_1} R_1(d)F_d(z) \big) \\[5pt] 
 \frac{1}{\gamma_1} F_d(z)                                                         
\end{bmatrix},
\end{aligned}
\end{equation}
with $F_{\bc}(z)=(F_1(z),\ldots, F_N(z))$, 
where $(F_k(z), F_v(z), F_d(z))$ are introduced in \eqref{Fc-Fv-Fd}.
It should be noted that for $F_v(z)$ we also replace $\partial_t d$ by 
$$\gamma_1\partial_t d= \Delta d + |\nabla d|^2 d + \varepsilon_a P(d) (\nabla \Phi \otimes \nabla\Phi) d  - \gamma_1 P(d)D(v)d +v\nabla d.$$ 
Hence  $F_v(z)$  contains spatial derivatives of $d$  up to second order and spatial derivatives of $v$ up to first order.

\smallskip\noindent
For simplicity, we will  introduce the notation
\begin{equation}
\label{AB}
\cA(z):= \sA(z)z  \quad \text{and} \quad \cB(z):=\sB(z)z.
\end{equation}
Then it follows from Proposition~\ref{app: pro: ABF-smooth} that
\begin{equation}
\label{ABF-smooth}
\cA,\sF\in C^\omega(\EE_{1,\mu}(J_T), \EE_{0,\mu}(J_T)),\quad \cB\in C^\omega(\EE_{1,\mu}(J_T), \FF_{\mu}(J_T)),
\end{equation}
with
\begin{equation}
\label{ABF-smooth-derivatives}
\cA'(z_*)z=\sA(z_*)z+[\sA'(z_*)z]z_*,\quad \cB'(z_*)z=\sB(z_*)z+[\sB'(z_*)z]z_*,
\end{equation}
where $z,z_*\in \EE_{1,\mu}(J_T)$.

\medskip\noindent
For the proof of local well-posedness of \eqref{nonlinear abstract},  we consider the following linearized non-autonomous problem
\begin{equation}
   \label{linearized}
   \left\{
    \begin{aligned}
      \pa_t z   + \cA'(z_*(t))z - \sF^\prime(z_*(t))z&={\sf f}(t)&&\text{in}&&\Omega, \\
                    \cB'(z_*(t))z  &={\sf g}(t)&&\text{on}&&\pa\Omega, \\
      z(0)&=z^0 &&\text{in}&&\Omega.
    \end{aligned}
    \right.
  \end{equation}
We show that the linear problem~\eqref{linearized} enjoys maximal regularity.
\begin{proposition}
\label{pro: linearized}
Suppose  that \eqref{positivity}  and \eqref{mu-condition} hold. Then the linearized  problem \eqref{linearized}
has for each $z_*\in \EE_{1,\mu}(J_T)$  a unique solution  $z\in \EE_{1,\mu}(J_T)$,
provided
\begin{equation*}
\begin{split}
({\sf f},{\sf g}, z^0) \in \widetilde{\DD}_\mu(z_*,T):&=\{\EE_{0,\mu}(J_T)\times \FF_\mu(J_T)\times X_{\gamma,\mu} : 
  \cB^\prime (z_*(0))z^0  ={\sf g}(0) \}.
\end{split}
\end{equation*}
In this case, there is a constant $c_2=c_2(T)>0$ such that
    \begin{equation*}
    \| z \|_{\EE_{1,\mu}(J_T)}\le c_2(T) \big( \| {\sf f} \|_{\EE_{0,\mu}(J_T)} + \| {\sf g} \|_{\FF_\mu(J_T)} + \|z^0\|_{X_{\gamma,\mu}}\big).
    \end{equation*}
   Given any $T_*>0$, the constant $c_2$ is independent of $T\in (0,T_*]$ in case ${\sf g}\in {_0}\FF_\mu(J_T)$ and $z^0=0$.\\
\end{proposition}
\begin{proof}
We partition the interval $\bar J_T=[0,T]$ into finitely many
equidistant  subintervals $ I_j = [t_j, t_{j+2}]$ such that
$$
t_0=0<t_1<\cdots<t_k=T,\quad 
$$
 Note that  $I_j\cap I_{j+1} = [t_{j+1}, t_{j+2}]$. 
We then show that  \eqref{linearized} can be successively solved on each subinterval $I_j$, provided 
that $ | I_j | :=(t_{j+2}-t_{j})$ is sufficiently small. 
Hence, we solve  \eqref{linearized} on short subintervals, 
and then piece together these local solutions to obtain a solution on the whole time interval.
The main ingredient will be Proposition~\ref{pro: abstract-linear-tilde} and a contraction mapping argument.

\medskip\noindent
We recast \eqref{linearized} as 
\begin{equation}
\label{linearized-j}
   \left\{
    \begin{aligned}
      \pa_t z   + \sA(z_*(t_j))z &={\sf f}(t) - \sR^1_j(t)z &&\text{in}&&\Omega, && t\in I_j , \\
                    \sB(z_*(t_j))z   &={\sf g}(t) -  \sR^2_j(t)z &&\text{on}&&\pa\Omega, &&t\in I_j , \\
      z(t_j)&=z^0_j &&\text{in}&&\Omega,
    \end{aligned}
    \right.
  \end{equation}
for $j=0,\ldots, k$, where
\begin{equation}
\label{R1 and R2}
\begin{aligned}
& \sR^1_j(t)z =(\sA(z_*(t))  -\sA(z_*(t_j)))z+ [\sA^\prime (z_*(t)) z]z_*(t) - \sF^\prime(z_*(t))z, \\
& \sR^2_j(t) z= (\sB(z_*(t))  -\sB(z_*(t_j)))z+ [\sB^\prime (z_*(t)) z]z_*(t).
\end{aligned}
\end{equation}
It follows from Proposition~\ref{smallness} that for any given $\eta>0$, there exist $(m+1)$ intervals $I_j$ such that
\begin{equation}
\label{R1 and R_2 small}
\| \sR^1_j (\cdot)  z\|_{\EE_{0,\nu}(I_j)}+ \| \sR^2_j(\cdot) z\|_{\FF_\nu(I_j)}\le \eta \|z\|_{\EE_{1,\nu}(I_j)}, \quad z\in {_0}\EE_{1,\mu}(J_T), \quad 0\le j\le m.
\end{equation}

\bigskip\noindent
We set $z^0_0 := z^0$, and define $z^0_j$ iteratively by $z^0_j=\hat z_{j-1}(t_{j})$, 
with  $\hat z_{j-1} $  the unique solution of \eqref {linearized-j} on the interval $I_{j-1} $ obtained in a previous step for $j=1,\ldots, m$.

\smallskip\noindent
In a first step, we consider \eqref{linearized-j} for $j=0$.  For this case, let
\begin{equation*}
\begin{aligned}
W(I_0): &= \{ w\in \EE_{1,\mu}(I_0): w(0)=z^0 \ \  \text{and}\ \   \cB^\prime (z_*(0))z^0= {\sf g}(0)\}.
\end{aligned}
\end{equation*}
For $w\in W(I_0)$, let $z=z(w) \in \EE_{1,\mu}(I_0)$ be the unique solution of
\begin{equation}
\label{linearized-tau}
   \left\{
    \begin{aligned}
      \pa_t z   + \sA(z_*(0)) z& = {\sf f}(t)-  \sR^1_0(t)w(t)  &&\text{in}&&\Omega, && t\in I_0,\\
                    \sB(z_*(0))z   & ={\sf g}(t)-  \sR^2_0(t)w(t)  &&\text{on}&&\pa\Omega, && t\in I_0,\\
      z(0)&=z^0 &&\text{in}&&\Omega,
    \end{aligned}
    \right.
  \end{equation}
which exists according to Proposition~\ref{pro: abstract-linear-tilde}, as one readily verifies that the compatibility condition
$$\sB(z_*(0))z^0= g(0)- \sR^2_0(0)w(0)$$
is satisfied. Indeed, we have
$$
\sB(z_*(0))z^0 + \sR^2_0(0)w(0)= \sB(z_*(0))z^0 + [\sB^\prime (z_*(0))w(0)]z_*(0)= \cB^\prime(z_*(0))z^0 = {\sf g}(0),
$$
where the last two assertions follow from the definition of $W(I_0)$. We then define 
$$K: W(I_0) \to \EE_{1,\mu}(I) \quad \text{by} \quad K(w):=z(w)$$
with $z(w)$ the unique solution of  \eqref{linearized-tau}. Note that $K:W(I_0) \to W(I_0)$ is a self-map. 
We will show that $K$ is a contraction.

\medskip\noindent 
Let $w_i\in W(I_0)$ be given, with $i=1,2$. A moment of reflection reveals that
$\hat z:= K(w_1)-K(w_2)$ solves the linear problem
\begin{equation*}
   \left\{
    \begin{aligned}
      \pa_t \hat z  + \sA(z_*(0)) \hat z& = - \sR^1_0(t)(w_1(t) - w_2(t)) &&\text{in}&&\Omega, && t\in I_0, \\
                    \sB(z_*(0))\hat z   & =- \sR^2_0(t)(w_1(t) - w_2(t))   &&\text{on}&&\pa\Omega, && t\in I_0, \\
      \hat z(0)& =0 &&\text{in}&&\Omega.
    \end{aligned}
    \right.
\end{equation*}
Let $E:=\{z_*(t): t\in \bar J_T\}.$ As $z_*\in BU\!C(\bar J_T; X_{\gamma,\mu})$, we infer that $E\subset X_{\gamma,\mu}$ is 
compact. 
By Proposition~\ref{pro: abstract-linear-tilde}, there exists a constant $M$, which is uniform for any value of $z_*(t)\in E$ and does not depend on $|I_0|$,
such that 
\begin{equation*}
\|  \hat z \|_{\EE_{1,\mu}(I_0)}=
\| K(w_1)-K(w_2)\|_{\EE_{1,\mu}(I_0)}
\le M \Big[  \| \sR^1_0 (w_1-w_2)\|_{\EE_{0,\mu}(I_0)} + \|\sR^2_0 (w_1-w_2)\|_{\FF_\mu(I_0)} \Big].
\end{equation*}
Choosing $\eta= 1/2M$ it follows from~\eqref{R1 and R_2 small} that 
\begin{equation*}
\| \sR^1_0 (w_1-w_2)\|_{\EE_{0,\mu}(I_0)} + \|\sR^2_0 (w_1-w_2)\|_{\FF_\mu(I_0)} \le \frac{1}{2M} \|w_1-w_2\|_{\EE_{1,\mu}(I_0)}.
\end{equation*}
By the contraction mapping principle, the system~\eqref{linearized-tau} has a unique solution $\hat z_0\in \EE_{1,\nu}(I_0)$.
Hence,  $\hat z_0\in \EE_{1,\nu}(I_0)$ is also the unique solution of  \eqref{linearized} on $I_0$.

\medskip\noindent
In a next step, we show that 
\begin{equation}
\label{linearized-1}
   \left\{
    \begin{aligned}
      \pa_t z   + \sA(z_*(t_1))z &={\sf f}(t) - \sR^1_1(t)z &&\text{in}&&\Omega, && t\in I_1 , \\
                    \sB(z_*(t_1))z   &={\sf g}(t) -  \sR^2_1(t)z &&\text{on}&&\pa\Omega, &&t\in I_1 , \\
      z(t_1)&=z^0_1 &&\text{in}&&\Omega,
    \end{aligned}
    \right.
  \end{equation}
 has a unique solution, where $z^0_1= \hat z_0(t_1)$.
 Let
$$W(I_1)=\{ w\in \EE_{1}(I_1) : w(t_1)= z^0_1 \},
\quad \text{where}\quad  z^0_1 = \hat z_0(t_1).$$
As $\hat z_0$ solves \eqref{linearized} on $I_0$, the compatibility condition $\cB^\prime(z_*(t_1))z^0_1= g(t_1)$ is satisfied.
For $w \in W(I_1)$  we define 
$$ K: W(I_1) \to W(I_1)$$
as above.
Analogous arguments as above show that $K$ is a contraction.
 (The constants that appear in the estimates satisfy the same bounds as in step 1.)
Let $\hat z_1\in W(I_1)$ be the unique fixed point of $K$, that is, the unique solution of \eqref{linearized-1}.

\smallskip
Hence, setting
\begin{equation*}
z(t):=
\left\{
\begin{aligned} &\hat z_0(t), &&   t  \in [t_0, t_1], \\
                          & \hat z_1(t),  &&  t \in [t_1, t_{3}],
\end{aligned}
\right.                       
\end{equation*}
we can conclude that $z\in \EE_{1,\mu}([t_0, t_3]) $ is the (unique) solution of \eqref{linearized} on the interval $[t_0, t_3]$.
By reiterating this procedure, we obtain a unique solution $z\in \EE_{1,\mu}(J_T)$.

\medskip\noindent
The remaining assertions follow from the open mapping theorem, and an extension-restriction argument as in~\cite[Proposition 4.1(b)]{DSS24}. 
\end{proof}

\section{Existence and regularity of solutions}\label{section:well-posed}
\noindent
In this section we formulate and prove our main result concerning existence, uniqueness, and regularity of solutions.
We recall definition~\eqref{X-gamma-mu}.
\begin{theorem}
\label{thm: main}
Assume \eqref{positivity} and \eqref{mu-condition}.  
Let 
$ 
\cM_\mu= \{z\in X_{\gamma,\mu}: \sB(z)z=0 \}.
$ 

\medskip\noindent
Then for any  $z^0\in \cM_\mu$, there exists $T=T(z^0)>0$ such that the system \eqref{main-PDE-2} has a unique solution
$$(z, \pi)\in\EE_{1,\mu}(J_T)\times L_{p,\mu}(J_T; \dot H^1_p(\Omega)).$$
The solution can be continued to a maximal solution $z=z(\cdot,z^0)$ on an interval $[0, T_+(z^0))$.
\begin{itemize}
\item
The solution enjoys the additional time regularity: 
$$
t^k \partial^k_t z(\cdot, z^0)\in H^1_{p,\mu}(J_T; X_0)\cap L_{p,\mu}(J_T; X_1),
$$
for any  $T\in (0, T_+(z^0))$ and $k\in\N$. Moreover,  
$$z(\cdot, z^0)\in C^\infty ((0,T_+(z^0));X_1)\cap C^\omega ((0,T_+(z^0));X_{\gamma,\mu}).$$
\item
 \eqref{main-PDE-2} generates a local Lipschitz continuous semiflow on $\cM_\mu$.
 \vspace{2mm}
 \item
 If $ | d^0|=1$, then $| d(t)|=1$ for each $t\in [0, T_+(z^0))$.
 \vspace{2mm}
 \item
 If $c^0_k > 0$ in $\bar\Omega$, then $c_k(t)>0 $ in $\bar\Omega$ as well for each $t\in [0,T_+(z^0))$. 
 
 \smallskip\noindent
 More precisely, the following property holds. Suppose $c^0_k\ge c$ on $\bar\Omega$  for some constant $c>0$. Let $T\in (0, T_+(z^0))$ be given.
Then $c_k$ satisfies
\begin{align*}
    c_k(t)\ge Ae^{-Bt}c>0 
\end{align*}
for all $t\in [0,T]$, for constants $A,B>0$ depending on parameters and bounds on 
$\Phi$ in $L^\infty([0,T];W^{1,\infty}(\Omega)).$
 \end{itemize}
\end{theorem}
\begin{proof}
For existence of solutions and the semiflow property we follow the proof of  \cite[Theorem 5.1]{DSS23}, see also the respective proofs of 
 \cite[Theorem~14]{LPS06} and \cite[Proposition~4.3.2]{Mey10}.
 The main ingredients will be Proposition~\ref{pro: abstract-linear-tilde} and  Proposition~\ref{pro: linearized}.
 We first show that \eqref{nonlinear abstract} has a unique solution $z$ with the properties asserted. 

\medskip\noindent

Let  $z^0\in \cM_\mu$ and $T_0>0$ be given.
By Proposition~\ref{pro: abstract-linear-tilde}, the problem 
\begin{equation*}
\left\{
\begin{aligned}
 \partial_t z + \sA(z^0) z &= 0 &&\text{in} \quad\Omega, \\
\sB(z^0) z & = 0 &&\text{on}  \quad \partial\Omega, \\
       z(0)& = z^0 &&\text{in} \quad \Omega,
\end{aligned}
\right.
\end{equation*}
has a (unique) solution $z_* \in \EE_{1,\mu}(J_{T_0})$.
We then set 
\begin{equation*}
\begin{aligned}
\sA_* (t) z  &=  \cA'(z_* (t)) z  - \sF^\prime(z_* (t))  z, \\
\sB_* (t) z  &= \cB'(z_* (t))z  , 
\end{aligned}
\end{equation*}
for $t\in J_{T_0}$, where the functions $(\cA, \cB,\sF)$ are defined in \eqref{F(z)} and \eqref{AB}.
With this notation and~\eqref{ABF-smooth-derivatives}, the nonlinear system \eqref{nonlinear abstract} can be rewritten as
\begin{equation}
\label{shifted-problem}
\left\{
\begin{aligned}
      \pa_t z+\sA_* (t)  z  &=    \cA^\prime(z_*)z -\cA(z) +\sF(z) - \sF'(z_*)z &&\text{in}&&\Omega, \\
      \sB_*(t)  z                &=    \cB^\prime (z_*)z - \cB(z)  &&\text{on}&&\pa\Omega,\\
      z(0)&=z^0  &&\text{in}&&\Omega.
    \end{aligned}
    \right.
\end{equation}
In a first step, we consider the linear problem
\begin{equation}
\label{linear pb-z*}
\left\{
\begin{aligned}
      \pa_t z+\sA_* (t)  w  &=    \cA^\prime(z_*)z_*-\cA(z_*) +\sF(z_*) - \sF'(z_*)z_* &&\text{in}&&\Omega, \\
      \sB_*(t)  w               &=    \cB^\prime (z_*)z_* - \cB(z_*)  &&\text{on}&&\pa\Omega,\\
      w(0)&=z^0  &&\text{in}&&\Omega.
    \end{aligned}
    \right.
\end{equation}
The compatibility condition
$
\sB_*(0)  z^0=  \cB^\prime (z^0)z^0 - \cB(z^0) 
$
is satisfied, as $\cB(z^0)=\sB(z^0)z^0=0$ by assumption.  
Therefore, Proposition~\ref{pro: linearized} implies that for any $T_0>0$, the linear system \eqref{linear pb-z*} has a unique solution $w\in \EE_{1,\mu}(J_{T_0})$.

\medskip\noindent
Fix $T_0, R_0>0$. 
For every $T\in (0,T_0]$ and $R\in (0,R_0]$, we define a closed set in $\EE_{1,\mu}(J_T)$ by
$$
\Sigma(T,R)=\{ z\in \EE_{1,\mu}(J_T): \|z-w\|_{\EE_{1,\mu}(J_T)} \leq R,\,  z(0)= z^0 \},\quad
\text{where $w$ solves \eqref{linear pb-z*}}.
$$
In a second step, given any $\hat{z}\in \Sigma(T,R)$, we consider the  linear system 
\begin{equation}
\label{fixed point pb-z*}
\left\{
\begin{aligned}
      \pa_t z+\sA_* (t)  z  &=  \cA'(z_*) \hat{z}   - \cA(\hat{z})  +\sF(\hat{z}) - \sF'(z_*) \hat{z}  &&\text{in}&&\Omega, \\
      \sB_*(t)  z     &= \cB'(z_*) \hat{z}  - \cB(\hat{z})  &&\text{on}&&\pa\Omega,\\
      z(0)&=z^0  &&\text{in}&&\Omega.
    \end{aligned}
    \right.
\end{equation}
As $\cB(z^0)=\sB(z^0)z^0=0$, the compatibility condition 
$
\sB_*(0)  z^0= \cB'(z^0) z^0   - \cB(z^0)  
$
is satisfied and we can infer from Proposition~\ref{pro: linearized} that \eqref{fixed point pb-z*} has a unique solution 
$$z:=  \cT(\hat{z}) ~\in~\EE_{1,\mu}(J_T).$$
It is clear that $z\in \Sigma(T,R)$ solves \eqref{shifted-problem}  iff $z$ is a fixed point of $\cT$. 
 
 \medskip\noindent
Employing~Proposition~\ref{pro: linearized}, one shows as in the proof of   \cite[Theorem 5.1(a)]{DSS23} that  
$$\cT : \Sigma(T,R) \to \Sigma(T,R)$$
is a self-map  and a contraction,  provided $R$ and $T$ are chosen small enough.

The contraction mapping principle then implies the existence of a 
 unique solution $z\in \Sigma(T,R)$ to \eqref{shifted-problem}, and hence to \eqref{nonlinear abstract}, on the time interval $[0,T]$.
A standard argument then yields that $z$ is also the unique solution in  $\EE_{1,\mu}(J_T)$.

\medskip\noindent
The existence of a maximal interval of existence  $[0,T_+(z^0))$ can be obtained in a standard way as in \cite[Corollary~5.1.2]{PrSi16}. 

\medskip\noindent
The additional time regularity can be established by using the so-called parameter trick in conjunction with the implicit function theorem.
For this, we fix $T\in (0,T_+(z^0))$ and $\epsilon\in (0,1)$ so small  that $(1+\epsilon)T<T_+(z^0)$.
Let $z\in \EE_{1,\mu}(J_T)$ be the unique solution of \eqref{nonlinear abstract} with initial value $z^0$ and
let $z_\lambda(t)=z(\lambda t)$ for $t\in [0,T]$.
Then $v=z_\lambda$ solves
\begin{equation}
\label{time regularity pb}
\left\{
\begin{aligned}
      \pa_t v+\sA_* (t)  v  &= \fF(\lambda, v) &&\text{in}&&\Omega, \\
      \sB_*(t)  v     &=  \fG( v )&&\text{on}&&\pa\Omega,\\
      v(0)&= z^0 &&\text{in}&&\Omega ,
    \end{aligned}
    \right.
\end{equation}
on $[0,T]$, where here
\begin{align*}
\sA_* (t)  v  &= \cA'(z (t)) v - \sF^\prime(z (t))  v , \quad \sB_* (t) v   = \cB'(z (t))v , \\  
 \fF(\lambda, v) &=  -\big( \lambda \cA(v)-\cA^\prime(z)v\big) + \lambda \sF(v) - \sF^\prime(z )v,  \\
\fG(v)  & =    -\big(\cB(v) - \cB^\prime(z) v\big).
\end{align*}
From $\fF(1,v)=-\big(\cA(v)-\cA^\prime(z)v\big) + \sF(v)-\sF^\prime(z)v$ and the definition of $G(v)$
one readily verifies that
\begin{equation}
\label{F-G-derivative}
\partial_2 \fF (1,z)=0, \quad \fG'(z)=0,
\end{equation}
where $\partial_2 \fF(1,z)$ denotes the Fr\'echet derivative of $\fF(1,v)$ with respect to the variable $v$, evaluated at $(1,z)$.
We now define $\cF: (1-\epsilon,1+\epsilon) \times \EE_{1,\mu}(J_T) \to \EE_{1,\mu}(J_T)$   by
$$
\cF(\lambda, v):= v - K(\lambda, v),
$$
 where $\hat v= K(\lambda, v)$ is the solution of
\begin{equation}
\label{v-hat}
\left\{
\begin{aligned}
\partial_t \hat v + \sA_*(t) \hat v&=  \fF(\lambda, v) &&\text{in}&&\Omega, \\
  \sB_*(t) \hat v     &=  \fG( v )&&\text{on}&&\pa\Omega,\\
      \hat v(0)&= \hat z_0 + \cR(z^0)\fG(v(0)) &&\text{in}&&\Omega .
    \end{aligned}
    \right.
\end{equation}
Here,  $ \hat{z}_0 = z^0- \cR(z^0)   \cB^\prime(z^0)z^0 $, where
$\cR(z_0)\in \cL( Y_{\gamma,\mu},  X_{\gamma,\mu})$
is a right inverse of $\cB^\prime(z^0)\in\cL(X_{\gamma,\mu}, Y_{\gamma,\mu})$, 
see Proposition~\ref{pro: right-inverse}. 
It follows that
$\cB^\prime(z^0) \hat z_0 =0 ,$
and hence,  the compatibility condition 
$$ \cB^\prime(z^0) (\hat{z}_0+  \cR(z^0)   \fG(v(0)) )= \fG(v(0))$$ 
is again satisfied in \eqref{v-hat}, so that $K(1,v)$ is meaningfully defined.
We now observe that 
$\cF(1,z)=0,$
as $z$ and $\hat v = K(1,z)$ solve the same problem, due to
$$
\hat v(0)= z^0- \cR(z^0)   \cB^\prime(z^0)z^0 + \cR(z^0)\fG(z(0)) = z^0.
$$  
Next, we note that by \eqref{F-G-derivative},  $\partial_2 \cF(1,z) ={\rm Id} \in {\rm Isom }\,(\EE_{1,\mu}(J_T))$.
Therefore, the implicit function theorem implies that there exist $\delta \in (0,\epsilon) $ and 
\begin{equation}
\label{Psi-analytic}
\Psi\in C^\omega( (1-\delta, 1+\delta), \EE_{1,\mu}(J_T))
\end{equation}
such that
$$
\cF(\lambda, \Psi(\lambda))=0,\quad \lambda \in (1-\delta, 1+\delta).
$$
Hence,  $ \Psi(\lambda) = K(\lambda, \Psi(\lambda))$ and we conclude that $v=\Psi(\lambda)$ solves
\begin{equation*}
\left\{
\begin{aligned}
      \pa_t v+\sA_* (t)  v  &= \fF(\lambda, v) &&\text{in}&&\Omega, \\
      \sB_*(t)  v     &=  \fG( v )&&\text{on}&&\pa\Omega,\\
      v(0)&= z^0- \cR(z^0)   \cB^\prime(z^0)z^0  + \cR(z^0) \fG( v(0) ) &&\text{in}&&\Omega .
    \end{aligned}
    \right.
\end{equation*}
From $\cF(1,z)=0$, we infer that $z=\Psi(1)$. We want to show that $\Psi(\lambda)=z_\lambda$.
To this end, notice that $z^0(\lambda):={\rm tr}_0 \Psi(\lambda)$ satisfies
\begin{equation*}
\begin{aligned}
z^0(\lambda) &=  z^0- \cR(z^0)   \cB^\prime(z^0)z^0  + \cR(z^0) \fG( z^0(\lambda) )  \\
 &=z^0-\cR(z^0)\big(\cB(z^0(\lambda)) - \cB^\prime(z^0)(z^0(\lambda)-z^0)\big).
\end{aligned}
\end{equation*}
By using the fact that $\cB(z^0)=0$, we further obtain
\begin{equation*}
\begin{split}
z^0(\lambda) -   z^0 
 =- \cR(z^0) \big(\cB(z^0(\lambda))-\cB(z^0) - \cB^\prime (z^0)(z^0(\lambda)-z^0))\big).
\end{split}
\end{equation*}
We can thus conclude  that
\begin{align*}
\| z^0(\lambda) -   z^0 \|_{X_{\gamma,\mu}}  
&\leq  \eps (\| z^0(\lambda) -   z^0 \|_{X_{\gamma,\mu}})\| z^0(\lambda) -   z^0 \|_{X_{\gamma,\mu}}\\
 &\leq  \eps (\| \Psi(\lambda) -  \Psi(1) \|_{\EE_{1,\mu}(J_T) } )\| z^0(\lambda) -   z^0 \|_{X_{\gamma,\mu}},
\end{align*}
where $\eps(r) \to 0$ as $r\to 0^+$.
By choosing $\delta $ so small that 
$$\sup\limits_{\lambda\in (1-\delta, 1+\delta) } \eps (\| \Psi(\lambda) -  \Psi(1) \|_{ \EE_{1,\mu}(J_T) } ) \leq  1/2,$$ 
we have $z^0(\lambda)=z^0$ for all $\lambda\in (1-\delta, 1+\delta) $. 
Thus, $\Psi(\lambda)$ solves \eqref{time regularity pb} and hence $\Psi(\lambda)=z_\lambda$. The differentiability of $\Psi$ implies 
$$\left(\frac{d}{d\lambda}\right)^k \Psi(\lambda)\Big|_{\lambda =1} = t^k \partial^k_t z \in \EE_{1,\mu}(J_T),\quad k\in \N.$$
It follows that 
$z\in H^k_p((\eps, T); X_1)$
for any $\eps \in (0,T)$ and any $k\in \N$. By Sobolev embedding, we conclude that $z\in C^\infty((\eps,T);X_1)$. As $\eps$ and $T\in (0, T_+(z^0))$ can be
chosen arbitrarily, we infer that
$$z(\cdot, z^0)\in C^\infty((0,T_+(z^0)); X_1).$$
To show analyticity in time, we consider an arbitrarily fixed $t_0\in (0,T)$. It  follows from the fact that the mappings 
$$
\EE_{1,\mu}(J_T) \hookrightarrow BU\!C (J_T) \underset{ {\rm tr}_{t_0}  } {\to  } X_{\gamma,\mu} ,\quad {\rm tr}_{t_0} v = v(t_0),
$$
are linear and bounded, and from \eqref{Psi-analytic} that
$ [\lambda \mapsto z(\lambda t_0)]\in C^\omega( (1-\delta, 1+\delta);  X_{\gamma,\mu}) $  
This, in turn, yields the assertion $z\in C^\omega((0, T); X_{\gamma,\mu}).$
\goodbreak
\medskip\noindent
The proof  that \eqref{nonlinear abstract} generates a Lipschitz continuous semiflow follows the same lines as the proof of \cite[Theorem 5.1(b)]{DSS23},
see also 
 \cite[Theorem~14]{LPS06} and \cite[Proposition~4.3.2]{Mey10}.
 
 \medskip\noindent
Having established that \eqref{nonlinear abstract} has a (unique) solution  $z=(\bc, v, d) \in \EE_{1, \mu}(J_T)$, 
it remains to show how to generate a solution for \eqref{main-PDE-2}.
From  \eqref{nonlinear abstract} we know that the velocity $v$ satisfies
\begin{equation}
\label{v-equation-1}
\partial_t v= \PP_H w \quad \text{on}\ \  J_T,
\end{equation}
where
$
w  = A_v(d)v -\gamma^{-1}_1 R_1(d) [\Delta d - R_0(d) v +F_d(z)]  +  F_v(z).
$
 
\smallskip\noindent
Using the relation \eqref{substitution}, which is valid as $z$ is a solution to \eqref{nonlinear abstract}, results in 
 \begin{equation}
\label{w}
w  = A_v(d)v -R_1(d) \partial_t d+  F_v(z).
\end{equation}
 Note that $w\in L_{p,\mu}(J_T; L_p(\Omega; \R^n))$.
For $t\in J_T$, let  $\pi(t)$ be the  solution to 
  $$
  (\nabla\pi(t)| \nabla\phi)_\Omega =( w(t)| \nabla\phi)_\Omega , \quad \phi\in \dot{H}_{p'}^1(\Omega).
  $$ 
  It follows that $\pi\in L_{p, \mu}( J_T;\dot{H}_p^1(\Omega))$. We then have
  $$
  \PP_H w(t)=w(t)-\nabla \pi(t), \quad t\in J_T,
  $$ 
  by definition of the Helmholtz projection.
It follows from \eqref{v-equation-1} that $\pa_t v= \PP_H w = w- \nabla \pi$
  and, hence,
  $ \pa_t v+ \nabla \pi = w. $
 With \eqref{w} we conclude that 
 $$
 \partial_t v  -A_v(d)v +R_1(d)\partial_t d +\nabla \pi   =  F_v(z), \quad t\in J_T.
 $$    
Consequently,  
 $(z,\pi)$ solves \eqref{main-PDE-3}, and hence also the system \eqref{main-PDE-2}, on $J_T$.
 
 \medskip\noindent
Next we will  show that the constraint $|d(t)|=1$ is preserved, provided it holds initially.
 Suppose $|d^0|=1$ and let $T\in (0, T_+(z^0))$ be fixed.
Let $\varphi(t)= |d(t)|^2 -1$ for $t\in [0,T]$.
Then  $\varphi$ has sufficient regularity to justify all the steps below.
Multiplying $\eqref{main-PDE-2}_6$ by $d$ and using the properties
 $$
 P(d)d = -\varphi d, \quad  (\Delta d + |\nabla d|^2 d)\cdot d= \frac{1}{2} \Delta \varphi + | \nabla d|^2 \varphi
 $$
we verify that $\varphi$ satisfies the following parabolic  equation 
\begin{equation}
\label{phi-equation}
\begin{aligned}
\frac{\gamma_1}{2} \Big(\partial_t \varphi +  v\cdot  \nabla \varphi \Big) -\gamma_2 ((\nabla v) d\cdot d)\varphi  -\frac{1}{2} \Delta \varphi - |\nabla d|^2 \varphi
+ \ve_a (\nabla \Phi \cdot d)^2 \varphi  &=0,  \\
 \partial_\nu \varphi &= 0, \\
 \varphi (0) & =0.
\end{aligned}
\end{equation}
We shall show that $\varphi=0$ on $[0,T]$.
Multiplying $\eqref{phi-equation}_1$ by $\varphi$  and integrating over $\Omega$ (using ${\rm div}\, v=0, v=0 $ on $\pa\Omega$) yields
\begin{equation*}
\begin{aligned}
&\frac{\gamma_1}{4} \frac{d}{dt} \int_\Omega \varphi^2\,dx 
+\frac{1}{2} \int_\Omega | \nabla \varphi |^2\,dx+ \ve_a \int_\Omega (\nabla \Phi \cdot d)^2 \varphi^2 \,dx  \\
& \qquad \qquad \qquad = \gamma_2 \int_\Omega ( (\nabla v) d\cdot d)\varphi^2 \,dx  +\int_{\Omega} |\nabla d|^2 \varphi^2 \,dx .
\end{aligned}
\end{equation*}
Hence, with $\psi(t) =  \int_\Omega \varphi^2(t)\,dx$ we obtain
$
\frac{d}{dt} \psi (t) \le  M \,\psi(t)$ for $t\in [0,T],$
where
$$
M=4\gamma^{-1}_1 \max_{(t,x)\in [0,T]\times\bar\Omega}\{\gamma_2| (\nabla v)d \cdot d|(t,x) + |\nabla d|^2(t,x) \}.
$$
The maximum indeed exists as $ v\in C([0,T], C^1(\bar\Omega; \R^n ))$ and $d \in C([0,T];C^2(\bar\Omega; \R^n))$.
By Gronwall's inequality,
$$\psi(t)\le \psi(0)\exp(M t), \quad t\in [0,T].$$
As $\psi(0)=0$, it follows that $\psi(t)=0$ on $[0,T]$. Since $T\in (0, T_+(z^0))$ can be chosen arbitrarily, the assertion follows.

\medskip\noindent
The assertion concerning positivity of concentrations can be obtained by analogous arguments as in the proof of 
\cite[Appendix B, Proposition 5]{Lee23}. For this, we observe that the matrices $\cD_k$ are positive definite by \eqref{D-positive-definite}
and $\Phi\in C ([0,T]; C^1(\bar\Omega))$ for each $T\in (0, T_+(z^0))$ by   \eqref{Ws-BUC} and Proposition~\ref{pro: Phi-time}.

\medskip\noindent
 The proof of Theorem~\ref{thm: main} is now complete.
\end{proof}
\begin{remark}
\label{rem: various formulations}
Our mathematical analysis concerning the existence of solutions and their properties, such as temporal regularity, the semiflow property, and global behavior, 
is based on system~\eqref{nonlinear abstract}. 
We have shown above that the pressure term $\pi$ can always be recovered once a solution $z=(\bc, v, d)$ of~\eqref{nonlinear abstract} has been obtained. 
This allows us to make statements about the solvability of~\eqref{main-PDE-2}, and consequently also of~\eqref{main-PDE}, taking Remark~\ref{rem: equivalent} into account.

On the other hand, the pressure term $\nabla \pi$ in~\eqref{main-PDE}, respectively~\eqref{main-PDE-2}, can be eliminated by applying the Helmholtz projection, which leads back to system~\eqref{nonlinear abstract}.

With this understanding, we will move freely between the various formulations of the model when making statements about solutions. 
Hence, even though all properties are established for the dynamic variable $z=(\bc, v, d)$, 
we will  take the liberty of including the pressure term and write $(z,\pi)$ when convenient.

\end{remark}


\section{Energy dissipation and equilibria}
\label{energy estimate}
\noindent
Let $(\bc, v, d,\pi)$ be the solution of \eqref{main-PDE} defined on its maximal interval of existence $[0, T_+(z^0))$, 
where $z^0 = (\bc^0, v^0, d^0) \in X_{\gamma,\mu}$
with  $ |d^0 |=1$ and $c^0_k>0 $ in $\bar\Omega$.  The energy functional for system  \eqref{main-PDE} is defined by
\begin{align}
\label{E}
    E(t) = \int_{\Omega}\left(\frac{1}{2}|v|^2+\frac{1}{2}|\na d|^2+\sum_k c_k\ln c_k+\frac{1}{2}\ve(d)\na\Phi\cdot\na\Phi\right)dx ,\quad t\in [0, T_+(z^0)).
\end{align}
For notational brevity we suppress the time variable $t$ on the right side. 
Then we have the following energy estimate:
\begin{proposition}
\label{energy-est}
\begin{equation*}
\begin{aligned}
 \frac{d}{dt}E(t)  
 & +  \alpha\sum_k \int_{\Omega}c_k|\na \mu_k|^2\,dx  +  \alpha_4 \int_\Omega |D(v)d|^2\,dx \\
&+\int_{\Omega}\left(\alpha_1(d\cdot D(v)d)^2+(\gamma_2+\alpha_2+\alpha_3)(\od\cdot D(v)d)+\beta |D(v)d|^2+\gamma_1|\od|^2\right)dx\le 0\\
    \end{aligned}
\end{equation*}
for $t\in [0, T_+(z^0)).$
\end{proposition}
\begin{proof}
The proof will be given in Appendix \ref{Appendix B}.
\end{proof}

The following characterization of equilibrium is motivated by experimentally observed configurations. For instance, such equilibria have been investigated in the electro‑osmotic setting  \cite{TCGLW16} and, more recently, in the context of high‑order nonlinear electrophoresis \cite{Rajabi24}.
\goodbreak
\begin{proposition}
\label{pro: Lyapunov-equilibria}
Suppose  \eqref{positivity} holds. Then
\begin{enumerate}
\item[(a)] $-E$ is a strict Lyapunov function for system \eqref{main-PDE}, or equivalently for \eqref{main-PDE-2}.
\vspace{2mm}
\item[(b)]
The set $\cE=\{(\bc, v,d, \pi )\}$ of  equilibria is characterized by
$$c_k = Z_k e^{-z_k \Phi}, \quad 1\le k\le N, \quad  v=0,$$ 
where $Z_k$ are (free) positive constants.
Moreover, $(\Phi, d)$  at equilibrium satisfy the following system of equations
\begin{equation*}
\left\{
\begin{aligned}
 -{\rm div}\, (\ve(d)\nabla \Phi)&=\Sigma_{k=1}^{N} z_k Z_k e^{-z_k \Phi }  &&\text{in} && \Omega,\\
       \Phi &= 0 &&\text{on} && \partial \Omega, \\
\Delta d + |\nabla d|^2 d &= -\ve_a P(d) (\nabla\Phi \otimes \nabla\Phi)d  &&\text{in} &&\Omega,\\
                                          | d |  &= 1  &&\text{in} &&\Omega,\\
                            \partial_\nu d &=0     &&\text{on} &&\pa\Omega, \\
\end{aligned} 
\right. 
\end{equation*}
\end{enumerate}
and $\pi$ is, up to a constant, given by
\begin{equation*}
  \pi  =  \sum_k c_k  + \frac{1}{2}\Big( \ve_{\perp} |\nabla \Phi |^2 +\ve_a ((\nabla\Phi \otimes \nabla\Phi)d)\cdot d -| \nabla d |^2 \Big). 
\end{equation*}
 \end{proposition}
\begin{proof}
(a) It follows from \eqref{positivity} and Young's inequality that
$$ 
(\gamma_2 + \alpha_2 + \alpha_3) (\od\cdot D(v)d)
\ge -  \frac{(\gamma_2 + \alpha_2 +\alpha_3)^2}{4\beta } |  \od |^2  -  \beta |D(v)d|^2 .
$$
Setting 
$\delta: = \gamma_1 - {(\gamma_2 + \alpha_2 +\alpha_3)^2}/{(4\beta )}
$
we can conclude that  
$$
\int_{\Omega}\left(\alpha_1(d\cdot D(v)d)^2+(\gamma_2+\alpha_2+\alpha_3)(\od\cdot D(v)d)+\beta |D(v)d|^2+\gamma_1|\od|^2\right)dx \ge \delta \int_\Omega| \od|^2\, dx.
$$
Suppose that $\frac{d}{dt} E(t) = 0$ for $t \in (a,b)$, for some interval $(a,b)\subset (0, T_+(z^0))$.

\smallskip
From  Proposition~\ref{energy-est}, Korn's inequality and $v|_{\pa\Omega}=0$ we infer that
$$  v=0,\quad  \nabla \mu_k =0\quad \text{and}\quad \od =0. $$
This, in turn, implies 
$$(\partial_t c_k, \partial_t v, \partial_t d)= (0,0,0),\quad t\in (a,b),$$
and, hence, we are at equilibrium. This shows that $-E$ is a strict Lyapunov function.

\medskip\noindent 
(b) At equilibrium, 
\begin{equation}
  \label{equilibrium-system}
	\left\{
  \begin{aligned}
    -{\rm div}\, (\ve(d)\nabla \Phi)&=\Sigma_{k=1}^{N} z_k c_k =\rho &&\text{in} && \Omega,\\
       \Phi &= 0 &&\text{on} && \partial \Omega, \\
   \nabla \pi &=- {\rm div}(\nabla d\odot \nabla d)  + {\rm div} \left(\ve(d) (\nabla \Phi\otimes \nabla \Phi) \right)  &&\text{in} && \Omega, \\ 
     \Delta d  + |\nabla d|^2 d &=   -\ve_a P(d) (\nabla \Phi\otimes \nabla\Phi)d &&\text{in} && \Omega,\\
    | d | & = 1 &&\text{in} && \Omega, \\
    \partial_\nu d &=0 &&\text{on} && \partial \Omega .\\
  \end{aligned}
  \right.
\end{equation}

\medskip\noindent
We will show that  the term 
$
- {\rm div}(\nabla d\odot \nabla d)  +  {\rm div} \left(\nabla \Phi\otimes \nabla \Phi)\ve(d)  \right) 
$
in $\eqref{equilibrium-system}_3$  indeed has a gradient structure at equilibrium. 
Let us then first consider the term

\begin{align*}
    {\rm div}\, ((\na\Phi\otimes\na\Phi)\varepsilon(d)).
\end{align*}
Setting $\ve = \ve (d)$ and expanding, we get (using the symmetry of $\nabla \Phi \otimes \nabla \Phi$ and $\ve(d) $)
\begin{equation*}
\begin{aligned}
 {\rm div}\, ((\na\Phi\otimes\na\Phi)\ve ) &=\partial_i(\partial_k \Phi \partial_j \Phi \ve_{ji})e_k \\
 &  = \pa_i(\varepsilon_{ij}\pa_j\Phi\pa_k\Phi)e_k 
 = \pa_i(\varepsilon_{ij}\pa_j\Phi)\pa_k\Phi e_k + \varepsilon_{ij}\pa_j\Phi(\pa_i\pa_k\Phi)e_k.
\end{aligned}
\end{equation*}
We observe that by $\eqref{equilibrium-system}_1$
$$\pa_i(\varepsilon_{ij}\pa_j\Phi)\pa_k\Phi e_k={\rm div}\,(\varepsilon(d)\na\Phi)\na\Phi= -\rho\na\Phi,$$ 
which, in the stationary case, is a gradient. 
Namely, we conclude from $\nabla \mu_k =0$ that 
$$\nabla c_k + z_k c_k \nabla \Phi =0.$$
Summing up yields
$-\rho \nabla \Phi = \nabla (\Sigma_i c_i)$
and hence
$$ \pa_i(\varepsilon_{ij}\pa_j\Phi)\pa_k\Phi e_k = \nabla (\Sigma_i c_i)$$
is a gradient term.

\medskip\noindent
Furthermore, since $\varepsilon_{ij} = \varepsilon_{\perp}\delta_{ij}+\varepsilon_a d_i d_j$, we have
\begin{equation}
\label{e2}
\begin{aligned}
\varepsilon_{ij}\pa_j\Phi(\pa_i\pa_k\Phi)e_k & = \varepsilon_{\perp}\pa_i\Phi(\pa_i\pa_k\Phi)e_k + \varepsilon_a d_i d_j\pa_j\Phi(\pa_i\pa_k\Phi)e_k \\
&= \frac{1}{2}\varepsilon_{\perp}\pa_k|\na\Phi|^2 e_k + \varepsilon_a d_i d_j\pa_j\Phi(\pa_i\pa_k\Phi)e_k  .\\
\end{aligned}
\end{equation}
So, $\varepsilon_{ij}\pa_j\Phi(\pa_i\pa_k\Phi)e_k$ can be decomposed into a gradient term and a nongradient term.

\medskip\noindent
Now, we take a look at the term $-{\rm div}\,(\na d\odot \na d)$. We have
\begin{equation*}
\begin{aligned}
-{\rm div}\,(\na d\odot\na d) = -\pa_i(\pa_i d\cdot\pa_j d)e_j & =-(\pa_i\pa_i d \cdot \pa_j d) e_j - (\pa_id\cdot\pa_i\pa_jd) e_j\\
& =-(\Delta d\cdot \pa_j d)e_j-\frac{1}{2}\pa_j | \nabla d |^2 e_j .\\
\end{aligned}
\end{equation*}
Next, we use the stationary equation for $d$, to calculate
\begin{equation*}
\begin{aligned}
-(\Delta d\cdot \pa_j d) &= \left[|\na d|^2d+\varepsilon_a P(d)(\na\Phi\otimes\na\Phi)d\right]\cdot \pa_j  d\\[2pt]
&=  \left[|\na d|^2d+\varepsilon_a (\na\Phi\otimes\na\Phi)d\right]\cdot \pa_j d \\[2pt]
& = |\na d|^2  d\cdot \pa_j d\ + \varepsilon_a (\na \Phi\otimes\na\Phi)d\cdot\pa_j d  \\[4pt]
& = \varepsilon_a (\na \Phi\otimes\na\Phi)d\cdot\pa_j d .
\end{aligned}
\end{equation*}
Here we used that $P(d) \partial_j d = \partial_j d$ and $d\cdot \partial_j d =0$ as  $ |d |=1$.
We have, by permuting indices,

\begin{equation*}
\begin{aligned}
(\na \Phi\otimes\na\Phi)d\cdot\pa_jd  
& = \pa_j((\na\Phi\otimes\na\Phi)d\cdot d) - \pa_j(\pa_i\Phi\pa_k\Phi d_k)d_i \\
& =\pa_j((\na\Phi\otimes\na\Phi)d\cdot d)-2\pa_j\pa_k\Phi \pa_i\Phi d_i d_k  - \pa_k\Phi\pa_i\Phi d_i \pa_j d_k \\
& =\pa_j((\na\Phi\otimes\na\Phi)d\cdot d)-2\pa_j\pa_k\Phi \pa_i\Phi d_i d_k  - (\na\Phi\otimes\na\Phi)d\cdot\pa_j d . 
\end{aligned}
\end{equation*}
Therefore, 
\be\label{e}
\bal
\varepsilon_a(\na \Phi\otimes\na\Phi)d\cdot\pa_jd\,e_j& =\frac{\varepsilon_a}{2}\pa_j((\na\Phi\otimes\na\Phi)d\cdot d)\,e_j-\varepsilon_a\pa_k\pa_i\Phi \pa_j\Phi d_i d_j \,e_k.
\eal
\ee
We note that the second term on the right hand side of \eqref{e} cancels out with the second term on the right hand side of \eqref{e2}. Thus, we are just left with gradient terms.

\medskip\noindent
It remains to observe that the relation $\nabla \mu_k=0$ implies $\mu_k = \ln c_k  + z_k\Phi = C_k$, where $C_k$ is a constant. 
This readily yields
$$
c_k = e^{C_k} e^{-z_k \Phi} =: Z_k e^{-z_k \Phi}, \quad 1\le k\le N. 
$$
\end{proof}
\goodbreak
\section{Long-time behavior of solutions}
\label{section:long-term}
\noindent
The next theorem   establishes the long-time behavior of solutions to  \eqref{main-PDE-2}.
\begin{theorem}\label{thm:global existence}
Assume \eqref{positivity} and \eqref{mu-condition}. 
Let $(z(\cdot, z^0),\pi)$ be the solution of \eqref{main-PDE-2}, defined on its maximal interval of existence $[0, T_+(z^0)).$ Then
\begin{enumerate}
\item[{\rm (a)}]  
\vspace{2mm}\noindent
The following alternatives hold:

\vspace{2mm}\noindent
Either  $T_+(z^0)=\infty$, that is, $z$ is a global solution, 

\vspace{2mm}\noindent
 or $\displaystyle \lim_{t\to T_+(z^0)}z(t) $ does not exist in $X_{\gamma,\mu}$.

\vspace{1mm}
\item[{\rm (b)}]
Suppose 
$$
\displaystyle \sup_{t\in [\eta,  T_+(z^0))} \|z(t)\|_{X_{\gamma, \overline{\mu}}}<\infty
\quad
\text{for some $\eta \in (0, T_+(z^0))$ and some $\bar{\mu}\in(\mu, 1]$}.
$$
Then the solution $z$ exists globally and ${\rm dist}\,(z(t), \cE)\to 0$ in $X_{\gamma, 1}$ as $t\to \infty$. 
\end{enumerate}
\end{theorem}
\begin{proof}
(a) 
We will prove the assertion by following the strategy in \cite[Corollary~5.1.2]{PrSi16}, see also the proof of \cite[Theorem 6.1(a)]{DSS24}. 
We argue by contradiction.
Suppose that $T_+(z^0)<\infty$ and $z(\cdot,z^0)$ converges to some $z_1$ in $X_{\gamma,\mu}$ as $t\to T_+(z^0)$.
Lemma~\ref{elliptic-tilde} implies that $\sB(z_1)z_1=0$, showing that $z_1\in \cM_\mu$. 

\smallskip
Therefore, the orbit $\cV:=\{z(t) : 0\le t <  T_+(z^0) \}$ is relatively compact in  $\cM_\mu$. 
It  follows from the property that \eqref{main-PDE-2} generates a  local semiflow on $\cM_\mu$  and a compactness argument that there exists 
a number $\delta>0$ such that  for each $s \in [0 , T_+(z^0))$, the system~\eqref{nonlinear abstract} with initial value 
$z(s)\in \cM_\mu$ has a unique solution in  $\EE_{1,\mu}(J_\delta)$.

\smallskip
Consequently, fixing $s_0\in (T_+(z^0)-\delta , T_+(z^0))$,  the system~\eqref{nonlinear abstract} with initial value $z(s_0)$ has a solution 
$\tilde z \in \EE_{1,1}(J_{\delta})$, which, by uniqueness, coincides with $z(s_0+ \cdot\, , z^0)$ on $[s_0, T_+(z^0))$.  
Hence, the solution  $z(\cdot, z^0)$ of \eqref{nonlinear abstract} can be extended beyond $T_+(z^0)$,
a contradiction to the maximality of $T_+(z^0).$

\medskip\noindent
(b)
We will prove the assertion by following the strategy in \cite[Section 5.7]{PrSi16}.
The proof is similar to the one of \cite[Theorem 6.1(b)]{DSS24}.  For the reader's convenience we include the details.
By Theorem~\ref{thm: main}, the system \eqref{nonlinear abstract} defines a local semiflow on $\cM_\mu $.
From the assumption   and the compact embedding
$$
X_{\gamma, \overline{\mu}} \hookrightarrow X_{\gamma, \mu},
$$
we infer that the orbit $\cV:=\{z(t) : 0\le t < T^+(z^0)\}$ is relatively compact in $\cM_\mu $.
Let $\overline{\cV}$ denote the closure of $\cV$ in $X_{\gamma,\mu}$.

\smallskip
It follows from a similar argument as in part (a) that
there exist a number $\delta>0$ and an open neighborhood  $\cU$ of $\overline{\cV}$ in $\cM_\mu $ such that for every $\tilde{z}_0\in \cU$,
 \eqref{nonlinear abstract} admits a unique solution $\tilde{z}\in \EE_{1,\mu}(J_{\delta})$. 
Moreover, the solution map $G_1:\cU\to \EE_{1,\mu}(J_{\delta})$ is continuous.
This implies  that, for any $t\in [0, T_+(z^0))$, the solution of \eqref{nonlinear abstract} with initial condition $z(t)$ exists on the  interval $[t,t+\delta]$, 
which further shows that $T_+(z^0)=\infty$.
Hence, the solution to \eqref{nonlinear abstract} is global.

As in the proof of Theorem~\ref{thm: main}, one shows that \eqref{nonlinear abstract} also defines a local semiflow on $\cM_1$, equipped with the metric induced by $X_{\gamma,1}$.
The inequality
\begin{align*}
\|z(\delta)\|_{X_{\gamma,1}} & \leq  \|z\|_{C([\delta/2,\delta]; X_{\gamma,1})} \leq C(\delta)
\|z\|_{\EE_{1,1}([\delta/2,\delta])} 
  \leq C(\delta) (\delta/2)^{ \mu -1} \|z\|_{\EE_{1,\mu}(J_{\delta})}   
\end{align*}
implies that the mapping $G_2: \EE_{1,\mu}(J_{\delta}) \to X_{\gamma,1}: [z \mapsto z(\delta)]$  is continuous. 
Consequently,  the composition map 
$$
G=G_2\circ G_1: \cU\to X_{\gamma,1}: \, z  \mapsto G_1(z)(\delta)$$
 is continuous as well.
We thus infer that the orbit $\{z(t)\}_{t\geq \delta}$ is relatively compact in $\cM_1$, as the continuous image  of a relatively compact set is again relatively compact.
Recall that the definition of $\omega$-limit set for \eqref{nonlinear abstract} is given by 
$$
\omega(z^0):=\{w\in X_{\gamma, 1}:\exists t_n\to \infty \text{ s.t. }\|z(t_n)-w\|_{X_{\gamma, 1}}\to 0 \text{ as }n\to \infty\}. 
$$
By \cite[Theorem~17.2]{Ama90}, $\omega(z^0)$ is nonempty, compact, connected in $\cM_1$, and 
\begin{equation}
\label{converg omega limit set}
\lim_{t\to \infty}{\rm dist}_{X_{\gamma, 1}}(z(t), \omega(z^0))=0 .
\end{equation}
It is shown in Proposition~\ref{pro: Lyapunov-equilibria} that $-E$ is a strict Lyapunov functional for \eqref{nonlinear abstract}.
Therefore,  
$\omega(z^0)\subset \cE$. Combining with \eqref{converg omega limit set}, this implies  $\lim_{t\to \infty}{\rm dist}_{X_{\gamma, 1}}(z(t), \cE)= 0.$
\end{proof}
\appendix

\section{Auxiliary results}
\label{Appendix A}
\noindent
 In this section, we  collect and prove results that are more of a technical nature.
 We start with studying existence and regularity of solutions for the time-dependent Poisson system.
 
\medskip\noindent
The following result will play a key role in establishing regularity properties of the boundary operator $\sB(z)$.
Before stating the result, we introduce some useful notation.
\begin{equation}
\label{Ed-Ec}
\begin{aligned}
\EE^{\bc}_{1,\mu}(J_T)&= H^1_{p,\mu}(J_T; L_p(\Omega; \R^N))\cap L_{p,\mu}(J_T; H^2_p(\Omega; \R^N)), \\
\EE^d_{1,\mu}(J_T)  &=H^1_{p,\mu}(J_T; H^1_p(\Omega; \R^n))\cap L_{p,\mu}(J_T; H^3_p(\Omega; \R^n)).
\end{aligned}
\end{equation}
\begin{proposition}
\label{pro: Phi-time}
Assume  that \eqref{permittivity} and \eqref{mu-condition} hold.
Then the time-dependent
elliptic boundary value problem 
    \begin{equation}
    \label{elliptic-phi-time}
    \left\{
    \begin{aligned}
    -{\rm div}\, (\ve( d(t))\nabla \Phi)&=\sum_{k=1}^N z_k c_k(t) &&\text{in} &&\Omega, \quad t\in [0,T],\\
           \Phi &= 0 &&\text{on} &&\partial \Omega, \\
    \end{aligned}
    \right.
    \end{equation}
admits for each $(\bc, d)\in \EE^{\bc}_{1,\mu}(J_T)\times \EE^d_{1,\mu}(J_T)$
 a unique solution $$\Phi =\Phi(\bc, d) \in W^{1/2-1/2p}_{p,\mu}(J_T; H^2_p(\Omega)).$$
Moreover, the mapping
\begin{equation}
\label{Phi-time-analytic}
[ (\bc , d) \mapsto \Phi (\bc, d)]: \EE^{\bc} _{1,\mu}(J_T)\times \EE^d_{1,\mu}(J_T) \to  W^{1/2-1/2p}_{p,\mu}(J_T; H^2_p(\Omega))
\end{equation}
is real analytic.  
\end{proposition}
\begin{proof}
It should be emphasized that the $W^{1/2-1/2p}_p$-time regularity of  $\Phi$ 
is not obvious, as $t$ only appears as a parameter in the elliptic problem \eqref{elliptic-phi-time}.

\smallskip\noindent
Let $(\bc,d)\in \EE^{\bc}_{1,\mu}(J_T)\times \EE^d_{1,\mu}(J_T)$ be given.
For $\bc\in \EE^{\bc}_{1,\mu}(J_T)$ one immediately verifies that
$$f:=\ell(c):= \sum_{k=1}^N z_k c_{k}  \in W^s_{p,\mu}(J_T; L_p(\Omega)),\quad s=1/2 -1/2p.$$
(In fact,  $f\in H^1_{p,\mu}(J_T; L_p(\Omega))).$
It follows from~\cite[Proposition 3.2]{MeSc12} that
$$
\EE^d_{1,\mu}(J_T)
\hookrightarrow  H^\sigma_{p,\mu} (J_T; H^{1+2(1-\sigma)}_p(\Omega; \R^n))
 $$
 for $\sigma\in [0,1]$.
This implies by Sobolev embedding and \cite[equation (2.1)]{MeSc12} 
\begin{equation*}
\EE_{1,\mu}(J_T)  \hookrightarrow  W^{\eta}_{p,\mu}(J_T; C^1(\bar\Omega; \R^n))\quad\text{for}\quad \eta< 1- n/2p.
\end{equation*}
Hence
\begin{equation}
\label{embedded-W-eta}
\EE^d_{1,\mu}(J_T)\hookrightarrow  W^{s_1}_{p,\mu}(J_T; C^1(\bar\Omega; \R^n))\hookrightarrow W^s_{p,\mu}(J_T; C^1(\bar\Omega;\R^n)),
\quad  s<s_1< 1- n/2p,
\end{equation}
where $s=1/2-1/2p$. This is feasible as $p>(n+2)$ by \eqref{mu-condition}.

We next observe that the mapping
\begin{equation*}
\label{d times d-analytic}
\begin{aligned}
&[d\mapsto \ve(d)]: 
 W^s_{p,\mu}(J_T; C^1(\bar\Omega;\R^n)) \to W^s_{p,\mu}(J_T ; C^1(\bar\Omega; \R^{n\times n})) \\
\end{aligned}
\end{equation*}
is real-analytic, as $W^s_{p,\mu}(J_T; C^1(\bar\Omega))$ is a continuous multiplication algebra.

\smallskip\noindent
For notational convenience, we now set
$$
a(t)= \ve (d(t)), \quad L(t)= {\rm div} (a(t)\nabla\cdot\ ), \quad t\in \bar J_T.
$$
According to Lemma~\ref{elliptic-tilde},  the elliptic problem \eqref{elliptic-phi-time} has for each $t\in \bar J_T$ a unique solution $\Phi(\bc(t), d(t) )$ which satisfies
\begin{equation*}
\label{Phi-time-parameter}
\end{equation*}
where $H^2_{p,D}(\Omega)= \{u\in H^2_p(\Omega) : {\rm tr}_{\pa\Omega} u=0\}$.
We aim to show that 
$[t\mapsto \Phi(\bc(t), d(t))]$ enjoys  $W^s_{p,\mu}$-time regularity. 
Using that 
$$
[(a,u) \mapsto au]: W^s_{p,\mu}(J_T; C^1(\bar\Omega; \R^{n\times n}))\times W^s_{p,\mu}(J_T; H^1_p(\Omega; \R^n))
		 \to W^s_{p,\mu}(J_T; H^1_p(\Omega; \R^n)) 
$$
is bilinear and continuous (and, hence, also real analytic), we obtain
\begin{equation*}
L(\cdot)\in \cL(W^s_{p,\mu}(J_T; H^2_p(\Omega)), W^s_{p,\mu}(J_T; L_p(\Omega))).
\end{equation*}
Moreover, 
$$
[t\mapsto L(t)]\in C(\bar J_T; \cL(H^2_p(\Omega),  L_p(\Omega))).
$$
Hence, by compactness and continuity of inversion, there exists a number $M>0$ such that 
\begin{equation}
\label{L(t)-inverse-bounded}
\| L(t)^{-1}\|_{\cL(L_p(\Omega), H^2_p(\Omega))}\le M\quad \text{for all } t\in \bar J_T.
\end{equation}

\smallskip\noindent
So far, we know that
$$ 
L(t)\in {\rm Isom}(H^2_{p,D}(\Omega), L_p(\Omega)), \quad t\in \bar J_T.
$$
In a next step, we show that
$$
L(\cdot) \in {\rm Isom}(W^s_{p,\mu} (J_T; H^2_{p,D} (\Omega)), W^s_{p,\mu}(J_T; L_p(\Omega))).
$$
For this, we observe that 
$$
L(t)^{-1} - L(\tau)^{-1} = - L(t)^{-1} (L(t) - L(\tau)) L(\tau)^{-1},\quad \tau, t \in J_T.
$$
For $\varphi \in L_p(\Omega)$, it then follows from~\eqref{L(t)-inverse-bounded} that
\begin{equation*}
\begin{aligned}
\| (L(t)^{-1} - L(\tau)^{-1})\varphi \|_{H^2_p(\Omega)} 
& \le M \| {\rm div}\, \big((a(t)-a(\tau))\nabla L(\tau)^{-1}\varphi \big) \|_{L_p(\Omega)} \\
&  \le CM \| \big((a(t)-a(\tau))\nabla L(\tau)^{-1}\varphi \big) \|_{H^1_p(\Omega)} \\
& \le CM^2 \| a(t) - a(\tau) \|_{C^1(\bar\Omega)} \| \varphi \|_{L_p(\Omega)}. 
\end{aligned}
\end{equation*}
Therefore,
\begin{equation}
\label{important-estimate}
\| L(t)^{-1} - L(\tau)^{-1} \|_{\cL (L_p(\Omega), H^2_p(\Omega))}\ \le CM^2 \| a(t) - a(\tau) \|_{C^1(\bar\Omega)},\quad \tau, t\in J_T.
\end{equation}

\smallskip\noindent
Suppose $f\in W^s_{p,\mu}(J_T; L_p(\Omega))$ and let $\Phi(t) := L(t)^{-1}f(t)$ for $t\in \bar J_T$. By \eqref{important-estimate}
\begin{equation*}
\begin{aligned}
[\Phi ]^p_{W^s_{p,\mu}(J_T; H^2_p(\Omega))}
&=\int^T_0 \int_0^t \tau^{(1-\mu)p} \frac{ || L(t)^{-1} f(t)- L(\tau)^{-1} f(\tau) \|^p_{H^2_p(\Omega)}} {(t-\tau)^{sp +1}}\, d\tau\, dt  \\[4pt] 
& \le (CM^2)^p  \sup_{t\in \bar J_T}\|f\|^p_{L_p(\Omega)}  [a]^p_{W^s_{p,\mu}(J_T; C^1(\bar\Omega))}  \\[4pt] 
&\quad + \sup_{\tau\in \bar J_T} \| L(\tau)^{-1}\|^p_{\cL(L_p(\Omega), H^2_p(\Omega))}  [f]^p_{W^s_{p,\mu}(J; L_p(\Omega))}  \\
& \le K\Big ([a]^p_{W^s_{p,\mu}(J_T; C^1(\bar\Omega))} +1 \Big) \| f \|^p_{W^s_{p,\mu}(J_T; L_p(\Omega))}
\end{aligned}
\end{equation*}
with a universal constant $K$. Moreover,
\begin{equation*}
\begin{aligned}
\| \Phi \|_{L_{p,\mu}(J_T; H^2_p(\Omega))} \le M \| f \|_{L_{p,\mu}(J_T; L_p(\Omega))}.
\end{aligned}
\end{equation*}

We have shown that the linear problem
$L(d) \Phi = f$ 
has for each given  $d\in \EE^d_{1,\mu}(J_T)$  and each
$f\in W^s_{p,\mu}(J_T; L_p(\Omega))$ a unique solution $\Phi\in W^s_{p,\mu}(J_T; H^2_{p,D}(\Omega))$, 
where $L(d)= {\rm div}(\eps(d)\nabla \,\cdot\, )$.
Hence,
$$
L(d)\in {\rm Isom} (W^s_{p,\mu}(J_T; H^2_{p,D}(\Omega)), W^s_{p,\mu}(J_T; L_p(\Omega)))$$
and
$$
\Phi=\Phi(\bc, d)=L( d)^{-1} \ell (\bc) \in W^s_{p,\mu}(J_T; H^2_p(\Omega)).$$ 
Since the mappings
\begin{equation*}
\begin{aligned}
&[d \mapsto L(d)]: \EE^d_{1,\mu}(J_T) \to 
 \cL (W^{s}_{p,\mu}(J_T; H^2_{p,D} (\Omega)), W^{s}_{p,\mu}(J_T; L_p(\Omega))),\\
&[\bc \mapsto \ell(\bc)]: \EE^{\bc}_{1,\mu}(J_T) \to  W^{s}_{p,\mu}(J_T; L_p(\Omega))
\end{aligned}
\end{equation*}
are real analytic
and  inversion is so as well, we  obtain \eqref{Phi-time-analytic}.  
\end{proof}
\begin{remark}
\label{rem: phi-analytic}
An inspection of the proof of Proposition~\ref{pro: Phi-time} reveals that we have in fact shown the following results: suppose  $I\subset \R$ is a compact interval. Then
\begin{itemize}
\item[(a)]
the elliptic problem 
\begin{equation*}
    \left\{
    \begin{aligned}
    -{\rm div}\, (a(\cdot) \nabla \Phi)&=f(t) &&\text{in} &&\Omega, \quad t\in I,\\
           \Phi &= 0 &&\text{on} &&\partial \Omega, \\
    \end{aligned}
    \right.
    \end{equation*}
has for each $a\in W^s_{p,\nu}(I; C^1(\bar\Omega; \R^{n\times n}))$ and $f\in W^s_{p,\nu}(I; L_p(\Omega))$ a unique solution
$$\Phi \in W^s_{p,\nu}(I; H^2_p(\Omega)),$$
provided  $ s\in (1- \nu + 1/p ,1)$.
This holds in particular if  $s=1/2-1/2p$,   $\nu =1$ or  $\nu=\mu$ with $\mu$ satisfying \eqref{mu-condition}. 
\vspace{2mm}
\item[(b)]
The mapping  $[(\bc, d)\mapsto \Phi(\bc , d)],$
\begin{equation*}
\begin{aligned}
 & W^s_{p,\nu}(J_T; L_p(\Omega; \R^N)) \times W^s_{p,\nu}( J_T; C^1(\bar\Omega; \R^n))\to  W^{s}_{p,\nu}(J_T; H^2_p(\Omega)),
\end{aligned} 
\end{equation*}
is real analytic, provided $s\in (1-\nu +1/p, 1)$. 
\end{itemize}
\end{remark}
\begin{proposition}
\label{app: pro: ABF-smooth}
Suppose  that \eqref{positivity}  and \eqref{mu-condition} hold and let
$(\cA, \cB, \sF)$ be as in \eqref{A(z)}, \eqref{F(z)}, ~\eqref{AB}.
 Then
$$
\cA,\sF\in C^\omega(\EE_{1,\mu}(J_T), \EE_{0,\mu}(J_T)),\quad \cB\in C^\omega(\EE_{1,\mu}(J_T), \FF_{\mu}(J_T)),
$$
and
$$
\cA'(z_*)z=\sA(z_*)z+[\sA'(z_*)z]z_*,\quad \cB'(z_*)z=\sB(z_*)z+[\sB'(z_*)z]z_*,
$$
where $z,z_*\in \EE_{1,\mu}(J_T)$.
\end{proposition}
\begin{proof}
We first consider the boundary operators $\cB$.
Let 
$$
\EE^1_{1,\mu}(J_T)=H^1_{p,\mu}(J_T; L_p(\Omega))\cap L_{p,\mu}(J_T; H^2_p(\Omega)).
$$
For $(\bc, d)\in \EE^{\bc}_{1,\mu}(J_T)\times \EE^d_{1,\mu}(J_T)$, see \eqref{Ed-Ec}, 
we have by the definition of $B_k$, see~\eqref {BCk},
$$ 
B_k(\textbf{c}, d)u = \nu \cdot {\rm tr}_{\partial\Omega} (\cD_k \nabla u) + \nu\cdot {\rm tr}_{\partial\Omega} (z_k\cD_k \nabla \Phi (\textbf{c}, d )) u.
$$
Proposition~\ref{pro: Phi-time} and trace theory imply
$$
\nu\cdot {\rm tr}_{\partial \Omega} (z_k \cD_k \nabla \Phi(\bc, d)) \in W^{1/2-1/2p}_{p,\mu}(J_T;  W^{1-1/p}_p(\partial\Omega)).
$$
It follows from the embedding
$$
W^{1/2-1/2p}_{p,\mu}(J_T;  W^{1-1/p}_p(\partial\Omega))\hookrightarrow \FF^1_\mu(J_T)
$$
and  the fact that $\FF^1_\mu(J_T)$ is a continuous multiplication algebra that the mapping
$$
 [u \mapsto \nu\cdot {\rm tr}_{\partial\Omega} (z_k\cD_k \nabla \Phi (\textbf{c}, d )) u]: \EE^1_{1,\mu}(J_T) \to \FF^1_\mu(J_T)
$$
is linear and bounded (and, thus, real analytic).
Moreover, it follows from  \cite{MeSc12}, Lemma 3.4 and Theorem  4.5, that
$$
 [u\mapsto  \nu \cdot {\rm tr}_{\partial\Omega} (\cD_k \nabla u)]: \EE^1_{1,\mu}(J_T) \to \FF^1_\mu(J_T)
$$ 
is linear and bounded as well. Hence, we can conclude from this and Proposition~\ref{pro: Phi-time} that
\begin{equation*}
[(\bc, d) \mapsto B_k(\bc, d)]: \EE^{\bc}_{1,\mu}(J_T)\times \EE^d_{1,\mu}(J_T)  \to \cL(\EE^1_{1,\mu }(J_T), \FF^1_\mu(J_T))
\end{equation*}
is real analytic. We then infer that
\begin{equation*}
\Big[(\bc, d) \mapsto  B_c(\bc,d) \bc =(B_1 (\bc, d)c_1, \ldots, B_N (\bc, d)c_N ) :
 \EE^{\bc}_{1,\mu}(J_T)\times \EE^d_{1,\mu}(J_T)  \to \FF_{\mu}(J_T)\Big]
\end{equation*}
 is real analytic as well. Hence, with the notational convention $\cB(z):= \sB(z)z= B_c(\bc, d)\bc$ the assertion of the proposition concerning $\cB(z)$ follows.

\bigskip\noindent
For further analysis, we recall that with \eqref{mu-condition} and~\eqref{interpolation},
\begin{equation}
\label{embedding-time}
\begin{aligned}
\EE_{1,\mu}(J_T) &\hookrightarrow BU\!C(\bar J_T;X_{\gamma,\mu}) \\
 & \hookrightarrow BU\!C(\bar J_T; C^1(\bar\Omega; \R^N))\times BU\!C(J_T; C^2(\bar\Omega; \R^n))\times BU\!C(J_T;  C^1(\bar\Omega; \R^n))).
\end{aligned} 
\end{equation}
Moreover, we observe that the multiplication mappings 
\begin{equation}
\label{bilinear}
\begin{aligned}
&[(u,w) \mapsto uw]: BU\!C(\bar J_T; C^j(\bar\Omega))\times BU\!C(J_T; H^j_{p} (\Omega))\to  BU\!C(J_T; H^j_p (\Omega)),\\ 
& [(u,w) \mapsto uw]: BU\!C(\bar J_T; C^j(\bar\Omega)) \times  BU\!C(\bar J_T; C^j(\bar\Omega)) \to  BU\!C(\bar J_T; C^j(\bar\Omega)), \\
&  [(u,w) \mapsto uw]: BU\!C(\bar J_T; H^1_p(\Omega)) \times  BU\!C(\bar J_T; H^1_p(\Omega)) \to  BU\!C(\bar J_T; H^1_p(\Omega)), 
\end{aligned}
\end{equation}
$j=0,1$, are bilinear and continuous and, hence, analytic.
Here we used that $p>n$, so that $H^1_p(\Omega)$ is a continuous multiplicative algebra.

We now consider the mapping $[z\mapsto \sA(z)]$.
Employing  \eqref{Av}-\eqref{R1} and \eqref{bilinear}, one readily verifies that 
 the mappings
\begin{equation*}
\begin{aligned} 
  [d\mapsto A_v(d)] & :  BU\!C(J_T; C(\bar\Omega;\R^n))\to \cL(\EE^v_{1,\mu}(J_T), L_{p,\mu}(J_T; L_p(\Omega; \R^n))),\\
  [d\mapsto R_0(d)] &: BU\!C(J_T; C^1(\bar\Omega;\R^n))\to \cL(\EE^v_{1,\mu}(J_T), L_{p,\mu}(J_T; H^1_p(\Omega; \R^n))), \\
 [d\mapsto  R_1(d)R_0(d)] &:  BU\!C(J_T; C^1(\bar\Omega;\R^n))\to \cL(\EE^v_{1,\mu}(J_T), L_{p,\mu}(J_T; L_p(\Omega; \R^n))) \\
  [d\mapsto R_1(d)\Delta ] &: BU\!C(J_T; C(\bar\Omega;\R^n)) \to \cL(\EE^d_{1,\mu}(J_T), L_{p,\mu}(J_T; L_p(\Omega; \R^n)))
\end{aligned}
\end{equation*}
are real analytic, where 
\begin{equation*}
\EE^v_{1,\mu}(J_T)= H^1_{p,\mu}(J_T; L_p(\Omega; \R^n))\cap L_{p,\mu}(J_T; H^2_p(\Omega; \R^n)). 
\end{equation*}
As $\PP_H\in \cL(L_p(\Omega; \R^n), L_{p,\sigma}(\Omega; \R^n))$ we infer that
\begin{equation*}
[d \mapsto \sA(d)]:  BU\!C(J_T; C^1(\bar\Omega; \R^n))\to \cL(\EE_{1,\mu}(J_T), \EE_{0,\mu}(J_T))
\end{equation*}
is real analytic, and we can conclude with~\eqref{embedding-time}  that 
\begin{equation}
\label{A-analytic}
[z\mapsto \sA(z)]\in C^\omega (\EE_{1,\mu}(J_T),  \cL (\EE_{1,\mu}(J_T), \EE_{0,\mu}(J_T))).
\end{equation}
This implies $\cA\in C^\omega(\EE_{1,\mu}(J_T), \EE_{0,\mu}(J_T))$ and the expression for  $\cA^\prime(z_*)z$ is then straightforward.

\bigskip\noindent 
Finally, we consider the nonlinear mapping $\sF$, and we first focus on the terms 
$$
{\rm div}\, (z_k c_k \cD_k \nabla \Phi (\bc,d)), \quad
{\rm div}\, (\ve (d) \nabla \Phi (\bc,d)\otimes \nabla \Phi(\bc,d)), \quad
\ve_a P(d)(\nabla \Phi (\bc,d)\otimes \nabla \Phi(\bc,d))d
$$
that occur in  \eqref{Fc-Fv-Fd} and \eqref{F(z)}.
We infer from  \eqref{Phi-time-analytic} and  \eqref{Ws-BUC}  that
$$
[(\bc,d) \mapsto \Phi (\bc, d)]\in C^\omega(\EE^d_{1,\mu}(J_T)\times \EE^d_{1,\mu}(J_T), BU\!C(\bar J_T; H^2_p(\Omega))),
$$
Hence, it follows from  \eqref{embedding-time}--\eqref{bilinear} that
\begin{align*}
 [(\bc,d)\mapsto {\rm div}\, (z_k c_k \cD_k \nabla \Phi (\bc,d)) ]& \in C^\omega(\EE_{1,\mu}(J_T), BU\!C(J_T; L_p(\Omega))), 
\\
 [(\bc ,d)\mapsto {\rm div}\, (\ve (d) \nabla \Phi (\bc,d)\otimes \nabla \Phi(\bc,d))] &\in C^\omega(\EE_{1,\mu}(J_T), BU\!C(J_T; L_p(\Omega; \R^n))), 
\\
 [(\bc, d)\mapsto \ve_a P(d)(\nabla \Phi (\bc,d)\otimes \nabla \Phi(\bc,d))d] &\in C^\omega(\EE_{1,\mu}(J_T), BU\!C(J_T; H^1_p(\Omega; \R^n))).
\end{align*}
Similar considerations as above, based again on the multilinear structure and  \eqref{embedding-time}--\eqref{bilinear}, then yield
$$
[z\mapsto (F_c(z), F_v(z), F_d(z))]
\in C^\omega(\EE_{1,\mu}(J_T), BU\!C(J_T; L_p(\Omega; \R^N)\times L_p(\Omega; \R^n)\times H^1_p(\Omega; \R^n))).
$$
It remains to observe that 
$$
[z\mapsto R_1(d)F_d(z)]\in C^\omega(\EE_{1,\mu}(J_T), BU\!C(J_T; L_p(\Omega ;\R^n))),
\quad
 \PP_H\in \cL(L_p(\Omega; \R^n), L_{p,\sigma}(\Omega; \R^n))$$ 
 to  conclude 
 \begin{equation}
 \label{BUC-X0}
 [(\bc, d, v)\mapsto \sF(\bc, d, v)] \in C^\omega( \EE_{1,\mu}(J_T), BU\!C(J_T; X_0)).
 \end{equation}
 The embedding $BU\!C(J_T; X_0)\hookrightarrow L_{p,\mu}(J_T;X_0)$ then yields the assertion
 $$
 [z\mapsto \sF(z)]\in C^\omega(\EE_{1,\mu}(J_T), \EE_{0,\mu}(J_T)).
 $$
\end{proof}
\begin{proposition}
\label{smallness}
Let $I_j=[t_j, t_{j+2}] $ be the subintervals of $[0,T]$ introduced in the proof of Proposition~\ref{pro: linearized}  and let 
$\sR^1_j(t)$ and $\sR^2_j(t)$ be as defined in  \eqref{R1 and R2}. 
Then for every $\eta>0$ there exists a number $\delta>0$ such that
\begin{equation*}
\| \sR^1_j (\cdot)  z\|_{\EE_{0,\nu}(I_j)}+ \| \sR^2_j(\cdot) z\|_{\FF_\nu(I_j)}\le \eta \|z\|_{\EE_{1,\nu}(I_j)}, \quad z\in {_0}\EE_{1,\nu}(I_j),
\end{equation*}
whenever $| I_j |\le \delta$.
Here we set $\nu=\mu $ in case $j=0$ and $\nu=1$ otherwise.
\end{proposition}
\begin{proof}
We first consider the terms $\sR^2_j(\cdot )z$.
For $(\bc, d) \in \EE^{\bc}_{1,\mu}(J_T)\times \EE^d_{1,\mu}(J_T)$, let
\begin{equation*}
\begin{aligned}
{\bf b}_k(\bc,d) :  & = z_k \cD_k \nabla \Phi(\bc, d). \\
\end{aligned}
\end{equation*}
where $\Phi(\bc, d)$ is the solution of~\eqref{elliptic-phi-time}.
By Remark \ref{rem: phi-analytic}, 
\begin{equation*}
[(\bc, d)\mapsto {\bf b}_k (\bc, d)]: W^s_{p,\nu}(I_j; L_p(\Omega; \R^N))\times W^s_{p,\nu} ( I_j; C^1(\bar\Omega; \R^n))
\to W^s_{p,\nu}(I_j; H^1_p(\Omega; \R^n))
\end{equation*}
is $C^1$ (in fact analytic). 
Hence, letting ${\bf b}^\prime_k(\bc_*, d_*)$ denote the Fr\'echet derivative of ${\bf b}_k$  at the point $(\bc_*,d_*)$, we obtain
\begin{equation}
\label{M-constant}
\|{\bf b}^\prime_k(\bc_*, d_*)(\bc,d)\|_{W^s_{p,\nu}(I_j; H^1_p(\Omega; \R^n))}
\le M \big[ \| \bc \|_{W^s_{p,\nu}(I_j; L_p(\Omega; \R^N)) } \!  + \!   \|d\|_{W^s_{p,\mu}(I_j; C^1(\bar\Omega; \R^n))} \big],
\quad 
\end{equation}
for $j=0,\ldots ,m$, $k=1,\ldots,N$, where the constant $M$  is independent of the intervals $I_j$.

\smallskip\noindent
With $ b_{k,\pa\Omega} (\bc, d): = \nu_{\pa \Omega} \cdot  {\rm tr}_{\pa \Omega} {\bf b}_k(\bc, d)$
we have
\begin{equation*}
B_k (\bc, d) c_k = \nu \cdot  {\rm tr }_{\pa\Omega} (\cD_k \nabla c_k) + b_{k,\pa \Omega}(\bc, d) {\rm tr}_{\pa \Omega} c_k
\end{equation*}
for $c_k \in \EE^1_{1,\mu}(J_T):= H^1_{p,\mu}(J_T; L_p(\Omega)) \cap L_{p,\mu}(J_T; H^2_p(\Omega)).$ 
To simplify notation, we fix $k\in \{1,\ldots, N\}$ and suppress it in the sequel. 
With this agreement, we obtain
\begin{equation*}
\begin{aligned}
  \sB(z)z &=\sB (\bc, d) c = \nu \cdot  {\rm tr }_{\pa\Omega} (\cD \nabla c)  +  b_{\pa \Omega}(\bc, d) {\rm tr}_{\pa \Omega} c,  \\
 [\sB^\prime(z_*)z ] z_*  & =[ b^\prime _{\pa \Omega }  (\bc_*, d_*)(\bc, d)] {\rm tr}_{\pa \Omega} c_* , 
\end{aligned}
\end{equation*}
where $b_{\pa \Omega}= b_{k, \pa\Omega}$, 
and where we set
$ b^\prime _{\pa \Omega }  (\bc_*, d_*)(\bc, d) = {\rm tr}_{\pa \Omega} [b^\prime   (\bc_*, d_*)(\bc, d)] $.


\medskip\noindent
We start with estimating the $W^s_{p,\nu}(I_j; L_p(\pa\Omega))$ semi-norm, $s=1/2-1/2p$, of the term
$$ 
\big(\sB(z_*(\cdot)) - \sB(z_*(t_j))\big)z = [b_{\pa\Omega}(\bc_*, d_*)(\cdot)-b_{\pa\Omega}(\bc_*, d_*)(t_j) ] {\rm tr}_{\pa\Omega} c
$$
in $\sR^2_j$ for $c\in {_0}\EE^1_{1,\nu} (I_j)$, to the result
\begin{align*}
& \big[\, [ b_{\pa\Omega}(\bc_*, d_*)(\cdot)-b_{\pa\Omega}(\bc_*, d_*)(t_j) ] {\rm tr}_{\pa\Omega} c \big]_{W^s_{p,\nu}(I_j; L_p(\pa \Omega))}^p  \\[4pt]
& \quad \leq   \sup_{t\in I_j}	\| b_{\pa\Omega}(\bc_*, d_*)(t)-b_{\pa\Omega}(\bc_*, d_*)(t_j)  \|^p_{L_\infty(\pa\Omega)) } 
\big[ {\rm tr}_{\pa\Omega} c \big]_{W^s_{p,\nu}(I_j; L_p(\pa \Omega))}^p  \\[4pt]
& \qquad +    \int_{I_j} \int^t_{t_j} \tau^{(1-\nu )p}  \| {\rm tr}_{\pa \Omega}c(\tau)\|_{L_p(\pa \Omega)}^p  
	\frac{\|  b_{\pa\Omega}(\bc_*,d_*)(t) -  b_{\pa\Omega}(\bc_*,d_*)(\tau) \|^p_{L_\infty(\pa \Omega)}}{ (t-\tau)^{sp+1}  } \, d\tau \, dt \\[4pt]
& \quad \le 
    C_{1,j}(z_*)  [{\rm tr}_{\pa\Omega} c]^p_{W^s_{p,\nu}(I_j; L_p(\pa \Omega))} 
 + C_{2,j}(z_*) \sup_{\tau \in I_j} \|{\rm tr}_{\pa\Omega} c (\tau)\|^p_{L_p(\pa \Omega)} \\[4pt]
  & \quad\le      C\Big (C_{1,j}(z_*)  + C_{2,j} (z_*) \Big) \|c \|^p_{\EE^1_{1,\nu}(I_j)} 
\end{align*}
where
\begin{equation*}
\begin{aligned}
C_{1,j}(z_*): &=  \sup_{ t\in I_j } \| b_{\pa \Omega} (\bc_*, d_*)(t)-b_{\pa \Omega}(\bc_*, d_*)(t_j)  \|^p_{L_\infty(\pa\Omega) } \\
C_{2,j} (z_*) : &= \int_{I_j} \int^t_{t_j} \tau^{(1-\nu )p}
 \frac { \|   b(\bc_*,d_*)(t) -   b(\bc_*,d_*)(\tau) \|^p_{H^1_p(\Omega)}}{ (t-\tau)^{sp+1} }  \,d\tau\,dt . \\
 \end{aligned}
 \end{equation*}
$C$ denotes the maximum of the respective norms of the continuous linear mappings
\begin{equation*}
\begin{aligned}
& {_0}\EE^1_{1,\nu}(I_j)= {_0}H^1_{p,\nu}(I_j; L_p(\Omega))\cap L_{p, \nu}(I_j; H^2_p(\Omega))
\underset{ {\rm tr}_{\pa \Omega} } {\longrightarrow  } {_0} W^s_{p,\nu}(I_j; L_p(\pa \Omega)) \\
&{_0}\EE^1_{1,\nu}(I_j) \hookrightarrow  {_0}BU\!C(I_j; H^1_p(\Omega))
 \underset{ {\rm tr}_{\pa \Omega} } {\longrightarrow  }{_0}BU\!C(I_j; L_p(\pa \Omega )) .
 \end{aligned}
 \end{equation*}
We know that $C$ is independent of $|I_j|$,  see for instance \cite[Theorem 4.5]{{MeSc12}} for the first line.
Moreover, we employed \eqref{mu-condition} to conclude that the mapping in the second line is defined and has norm independent of $| I_j |$.

\medskip\noindent
We  now show that  the quantities $C_{1,j}(z_*)$ and $C_{2,j}(z_*)$ can be rendered as small as we please by making $| I_j | $ small.

\smallskip
Indeed, we know from Proposition~\ref{pro: Phi-time}  that
$b(\bc_*, d_*) \in W^s_{p,\mu}(J_T; H^1_p(\Omega))$.
Moreover, we have the embeddings, see \eqref{mu-condition} and \cite[Proposition 2.10]{MeSc12},
$$
W^s_{p,\mu}(J_T; H^1_p(\Omega) ) \hookrightarrow BU\!C(J_T; H^1_p(\Omega)) \underset{{\rm tr}_{\pa\Omega} } {\longrightarrow} BU\!C(J_T ; L_\infty(\pa \Omega)).
$$
Hence, by uniform continuity, 
$C_{1,j}(z_*)$  can be made as small as we wish, by making $| I_j |$ sufficiently small. 
For the quantity $C_{2,j}(z_*)$ we again note that 
$b(\bc_*, d_*) \in W^s_{p,\mu}(J_T; H^1_p(\Omega))$. 
Hence, given any $\eta>0$, by absolute continuity of the double integral in~\eqref{seminorm} there exists $\delta>0$ such that
  $C_{2,j}(z_*)\le \eta$ whenever $| I_j | \le \delta$.

\medskip\noindent
We also have the estimate
\begin{align*}
&\| [b_{\pa\Omega}(\bc_*, d_*)(\cdot) - b_{\pa\Omega}(\bc_*,d_*)(t_j)] {\rm tr}_{\pa \Omega}c\|_{L_{p,\nu}(I_j; W^{1-1/p}_p(\pa\Omega))} \\
&\le  C\sup_{t\in I_j} \| b(\bc_*, d_*)(t) - b(\bc_*,d_*)(t_j) \|_{H^1_p(\Omega)} \| c\|_{L_{p,\nu}(I_j; W^1_p(\Omega))} \\
& \le C \sup_{t\in I_j} \| b(\bc_*, d_*)(t) - b(\bc_*,d_*)(t_j) \|_{H^1_p(\Omega)}  \| c\|_{\EE_{1,\nu}(I_j).}
\end{align*}
with a uniform constant $C$.
By uniform continuity, for given $\eta>0$ this term can be estimated by 
$\eta \| c\|_{\EE_{1,\nu}(I_j)}$, provided $| I_j |$ is chosen small enough.

\medskip\noindent
We have shown that for a given $\eta>0$, there exists a number $\delta>0$ such that if $| I_j |\le \delta$, 
$$
\|\, [ b_{\pa\Omega}(\bc_*, d_*)(\cdot)-b_{\pa\Omega}(\bc_*, d_*)(t_j) ] {\rm tr}_{\pa\Omega} c \,\|_{\FF_\nu(I_j)} \le \eta \|c\|_{\EE_{1,\nu}(I_j)}.
$$

\medskip\noindent
Next we observe that
\begin{align*}
& \big[{\rm tr}_{\pa \Omega}[b^\prime (\bc_*, d_*)  (\bc, d)] c_* \big]_{W^s_{p,\nu}(I_j; L_p(\partial \Omega))}^p  \\ 
& \quad  \leq \sup_{t \in I_j }  \| {\rm tr}_{\pa \Omega} b^\prime(\bc_*, d_*)(\bc, d) (t) \|_{L_\infty(\pa \Omega)}^p  
\int_{I_j} \int^t_{t_j} \tau^{(1-\nu )p}  
		\frac{\|  {\rm tr}_{\pa \Omega} (c_*(t) -  c_*(\tau)) \|^p_{L_p(\pa\Omega)}}{ (t-\tau)^{sp+1} } \, d\tau \, dt  \\[4pt]
& \quad +   \sup_{\tau \in I_j } \| {\rm tr}_{\pa \Omega} c_*(\tau )\|^p_{L_\infty(\pa \Omega)} 
\big[ {\rm tr}_{\pa\Omega} b^\prime(\bc_*, d_*)(\bc, d) \big]^p_{W^s_{p,\nu}(I_j; L_p(\partial \Omega))}  \\[4pt]
& \quad \le
C_{3,j} (z_*)\left \| b^\prime(c_*, d_*)(\bc, d )\right\|^p _{W^{s}_{p,\nu}(I_j; H^1_p(\Omega))} \\[4pt]
& \quad \le 
M  C_{3,j}(z_*) \big[ \| \bc \|_{W^s_{p,\nu}(I_j; L_p(\Omega)) } + \|d\|_{W^s_{p,\nu}(I_j; C^1(\bar\Omega))} \big]^p.
 \end{align*}
Here,   $M$ is the constant from \eqref{M-constant}, 
 \begin{equation*}
 \begin{aligned}
 C_{3,j} (z_*) & :=  C\Big[\sup_{\tau \in I_j } \|{\rm tr}_{\pa\Omega} c_*(\tau )\|^p_{L_\infty(\pa \Omega)} 
 +  \int_{I_j} \int^t_{t_j} \tau^{(1-\nu )p}  
		\frac{\|  {\rm tr}_{\pa \Omega} (c_*(t) -  c_*(\tau)) \|^p_{L_p(\pa\Omega)}}{ (t-\tau)^{sp+1} } \, d\tau \, dt \Big],\\
 \end{aligned}
 \end{equation*}
and $C$ is the maximum of the norms of the linear mappings
\begin{equation*}
\begin{aligned}
&{_0}W^s_{p,\mu}(I_j; H^1_p(\Omega)) \longrightarrow {_0}BU\!C(I_j; H^1_p(\Omega)) \underset{{\rm tr}_{\pa\Omega} } {\longrightarrow} {_0}BU\!C(I_j; L_\infty(\pa \Omega)), \\
& W^s_{p,\mu}(I_j; H^1_p(\Omega) ) \underset{{\rm tr}_{\pa\Omega} } {\longrightarrow}W^s_{p,\mu} (I_j; L_p(\partial\Omega)). \\
\end{aligned}
\end{equation*}
As before, $C$ does not depend on $| I_j |$.

\bigskip\noindent
Let $\eta>0$ be given. We want to show that there exists a number $\delta>0$ such that 
\begin{equation}
\label{app: NC}
M  C_{3,j}(z_*) \big[ \| \bc \|_{W^s_{p,\mu}(I_j; L_p(\Omega)) } + \|d\|_{W^s_{p,\nu}(I_j; C^1(\bar\Omega))} \big]^p 
\le \eta \| (\bc, d) \|^p_{\EE_{1,\nu}(I_j)} .
\end{equation}
whenever $|I_j | \le \delta$.
Here we note that the term $\sup_{\tau \in I_j} \|{\rm tr}_{\pa\Omega}c_*(\tau)\|_{L_\infty(\pa\Omega)}$ in the quantity $C_{3,j}(z_*)$ cannot be made small.
In fact, we can only deduce that it is uniformly bounded (by $\sup_{\tau \in \bar J_T} \|c_*(\tau)\|_{L_\infty(\pa\Omega)}$). 
While the second term in $C_{3,j}(z_*)$ can be made small for small $|I_j|$ by absolute continuity, we will only need to use the  bound
$[{\rm tr}_{\pa\Omega} c_*]_{W^s_{p,\mu}(J_T; L_p(\pa\Omega))}$. 
The assertion in \eqref{app: NC} can be shown by the following considerations.

\bigskip
For $\bc \in \EE^{\bc}_{1,\nu}(I_j)$
we obtain by H\"older's inequality 
\begin{equation*}
\begin{aligned}
\| \bc(t)-\bc(\tau) \|_{L_p(\Omega)}  
\le  \int^t_\tau \|\dot \bc (s)  \|_{L_p(\Omega)}  \,ds 
&\le \left(\int^t_\tau s^{(\nu-1)p^\prime}\,ds \right)^{1/{p^\prime}} \| \bc \|_{\EE^{\bc}_{1,\nu}(I_j) } \\
&\le c(\nu ,p) (t-\tau)^{\nu -1/p} \| \bc \|_{\EE^{\bc}_{1,\nu}(I_j)}.
\end{aligned}
\end{equation*}
Hence
\begin{align*}
[\, \bc \,]^p_{W^s_{p,\nu}(I_j; L_p(\Omega))}
& =\int_{I_j} \int^t_{t_j} \tau^{(1-\nu)p} \frac{ \|c(t) -c(\tau)\|^p_{L_p(\Omega)}} { (t-\tau)^{sp+1}} \,d\tau\,dt  \\
& \le c(\nu,p)^p \int_{I_j} \int^t_{t_j} \tau^{(1-\nu)p}  (t-\tau)^{\sigma -1} \,d\tau\,dt  \, \| \bc \|^p_{\EE^{\bc}_{1,\nu}(I_j)},
\end{align*}
where
$\sigma =(\nu-s)p-1$. It follows  from \eqref{mu-condition} that $\sigma>0$ as $n\ge 2$.
One readily verifies that given any $\eta>0$ there exists a number $\delta >0$ such that
$$
c(\nu,p)^p \int_{I_j} \int^t_{t_j} \tau^{(1-\nu)p}  (t-\tau)^{\sigma -1} \,d\tau\,dt \le \eta
\quad\text{whenever $| I_j |\le \delta$}.
$$ 
Moreover, we have for $\bc \in {_0}\EE^{\bc}_{1,\nu}(I_j)$ 
\begin{equation*}
\begin{aligned}
 \| \bc \|_{L_{p,\nu}(I_j; L_p(\Omega; \R^N))}
 & \le c(\nu,p) | I_j |^{1-\nu +1/p}  \| \bc  \|_{{_0}BUC(I_j; L_p(\Omega; \R^n))}  \\
& \le c_2  c(\nu,p) | I_j |^{1-\nu +1/p}  \| \bc \|_{{_0}\EE^c_{1,\nu}(I_j)},
\end{aligned}
\end{equation*}
where the embedding constant $c_2$ of ${_0}\EE^{\bc}_{1,\nu}(I_j)\hookrightarrow {_0}BU\!C(I_j ; L_p(\Omega; \R^N))$ is independent of $| I_j |$.
Hence, given any $\eta>0$, there exists $\delta>0$ such that
$$
\| \, \bc \, \|^p_{W^s_{p,\nu}(I_j; L_p(\Omega))}\le \eta  \| \bc \|^p_{\EE^{\bc}_{1,\nu}(I_j)},\quad\text{whenever $| I_j |\le \delta$.}
$$
To deal with the  term $\|d\|_{W^s_{p,\nu}(I_j; C^1(\bar\Omega; \R^n))}$ in~\eqref{app: NC}, we choose $s_1\in (s, 1-n/2p)$ so that 
 \eqref{embedded-W-eta} holds. 
 Then we have
$$ 
 {_0}\EE^d_{1,\nu}(I_j)\hookrightarrow {_0}W^{s_1}_{p,\nu}(I_j; C^1(\bar \Omega; \R^n))\hookrightarrow  {_0}W^s_{p,\nu}(I_j; H^1_p( \Omega; \R^n))
$$
and
$$ 
\| d \|_{{_0}W^{s_1}_{p,\nu}(I_j; H^1_p(\Omega; \R^n))}\le c_1 \|d\|_{{_0}\EE^d_{1,\nu}(I_j)},
$$
where the embedding constant $c_1$ is independent of $| I_j |$ for $d\in {_0}\EE^d_{1,\nu}(I_j)$.
One then readily verifies that
$$
[d]_{W^s_{p,\nu}(I_j; C^1(\bar \Omega, \R^n))}
\le | I_j |^{s_1 -s}[d]_{W^{s_1}_{p,\nu}(I_j; C^1(\bar \Omega, \R^n))}
\le c_1  | I_j |^{s_1 - s}  \|d\|_{\EE^d_{1,\nu}(I_j)}.
$$
As above, we have
\begin{equation*}
\begin{aligned}
 \| d \|_{L_{p,\nu}(I_j; C^1(\bar\Omega; \R^n))}
 & \le c(\nu,p) | I_j |^{1-\nu +1/p}  \| d \|_{{_0}BUC(I_j; C^1(\bar\Omega; \R^n))}  \\
& \le c_2  c(\nu,p) | I_j |^{1-\nu +1/p}  \| d \|_{{_0}\EE^d_{1,\nu}(I_j)},
\end{aligned}
\end{equation*}
where the embedding constant $c_2$ of ${_0}\EE^d_{1,\nu}(I_j)\hookrightarrow {_0}BU\!C(I_j, C^1(\bar\Omega; \R^n))$ is independent of $I_j$.

\medskip\noindent
We can now conclude that for every $\eta>0$ there exists a number $\delta>0$ such that
$$
 \| \bc \|_{W^s_{p,\nu}(I_j; L_p(\Omega)) } + \|d\|_{W^s_{p,\nu} C^1(\bar\Omega)^n)} \le \eta \|(\bc,d)\|_{\EE^{\bc}_{1,\nu}(I_j)\times \EE^d_{1,\nu}(I_j)},
 $$
for  $(\bc ,d)\in {_0}\EE^{\bc}_{1,\nu}(I_j)\times {_0}\EE^d_{1,\nu}(I_j)$, whenever $| I_j |\le \delta$.
It is not difficult to see that 
$$
  \| {\rm tr}_{\pa \Omega}[b^\prime (\bc_*, d_*)  (\bc, d)] c_* \|_{L_{p,\nu}(I_j; W^{1-1/p}_p(\partial \Omega))}
  \le   \eta \|(\bc,d)\|_{\EE^{\bc}_{1,\nu}(I_j)\times \EE^d_{1,\nu}(I_j)}.
$$
In summary, we have shown that given $\eta>0$ there exists $\delta >0 $ such that
\begin{equation*}
\| \sR^2_j(\cdot) z\|_{\FF_\nu(I_j)}\le \eta \|z\|_{\EE_{1,\nu}(I_j)}, \quad z\in {_0}\EE_{1,\nu}(I_j),
\end{equation*}
provided $ | I_j | \le \delta$.

\bigskip\noindent
For the term  $(\sA(z_*(\cdot))  -\sA(z_*(t_j))z$  in $\sR^1_j(\cdot )$ we see that
given any $\eta>0$ there exists $\delta>0$ such that
\begin{equation*}
\begin{aligned}
 \| \big (A_v(d_*(\cdot)) - A_v(d_*(t_j))\big)v \|_{L_{p,\nu}(I_j; L_p(\Omega; \R^n))}  &\le \eta \| v \|_{\EE^v_{1,\nu}(I_j)} ,\\
  \| \big( R_1(d_*(\cdot)) R_0(d_*(\cdot))-R_1(d_*(t_j)) R_0(d_*(t_j)) \big) v  \|_{L_{p,\nu}(I_j; L_p(\Omega; \R^n))}  &\le \eta \| v \|_{\EE^v_{1,\nu}(I_j)} 
\end{aligned}
\end{equation*}
provided $| I_j |\le \delta$. 
Indeed, this follows as $A_v(d)$ and $R_1(d) R_0(d)$ are second order differential operators 
with coefficients in $BU\!C(\bar J_T; C(\bar\Omega; \R^{n\times n}))$, acting on $v$,
see \eqref{Al-description} and \eqref{R0}-\eqref{R1}.
 (In fact, the coefficients have more spatial regularity, as $d\in BU\!C(J_T; C^2(\bar\Omega; \R^n)))$ but we do not need this here). 

For the same reason, given $\eta>0$ there exists $\delta>0$ such that
\begin{equation*}
\begin{aligned}
\| \big(R_1 (d_*(\cdot))- R_1(d_*(t_j))) \Delta d \|_{L_{p,\nu}(I_j; L_p(\Omega; \R^n))} & \le \eta \| d \|_{\EE^d_{1,\nu}(I_j)} , \\
 \| \big(R_0(d_*(\cdot)) - R_0(d_*(t_j))\big) v  \|_{L_{p,\nu}(I_j; H^1_p(\Omega; \R^n))} & \le \eta \| v \|_{\EE^v_{1,\nu}(I_j)}, 
\end{aligned}
\end{equation*}
provided $| I_j | \le \delta $.
In summary, we have shown that given $\eta>0$ there exists $\delta>0$ such that
\begin{equation*}
\| \big( \sA(z_*(\cdot))- \sA(z_*(t_j))\big) z \|_{\EE_{0,\nu}(I_j)}\le \eta \| z \|_{\EE_{1,\nu}(I_j)}. 
\end{equation*}
provided $| I_j | \le \delta$.

\medskip\noindent
For the term $ [\sA^\prime (z_*) z]z_*$ we have by \eqref{A-analytic}
\begin{equation*}
\begin{aligned}
  \| [\sA^\prime (z_*) z]z_*\|_{L_{p,\nu}(I_j; X_0)}
&  \le \| \sA^\prime (z_*)z\|_{\cL (\EE_{1,\nu}(I_j),L_{p,\nu}(I_j; X_0))} \|z_*\|_{L_{p,\nu}(I_j; X_0)} \\
 & \le C(z_*) \left( \int_{I_j} t^{(1-\nu)p } \| z_*(t)\|^p_{X_1}\, dt\right)^{1/p}\|z\|_{\EE_{1,\nu}(I_j) }. 
\end{aligned}
\end{equation*}
By absolute continuity of the integral $\int_{J_T}  t^{(1-\mu)p } \| z_*(t)\|^p_{X_1}\, dt $
we can find for any $\eta>0$ a number $\delta>0$ such that
$$
  \| [\sA^\prime (z) z]z_*\|_{L_{p,\nu}(I_j; X_0)}\le \eta  \|z\|_{\EE_{1,\nu}(I_j},\quad z\in {\EE_{1,\nu}(I_j)},
 $$
provided $| I_j | \le \delta$.
\medskip\noindent
Finally, for the term $\sF^\prime(z_*)z$ in $\sR^1_j(\cdot)z$ we  conclude with \eqref{BUC-X0} that
\begin{equation*}
\begin{aligned}
\| \sF^\prime (z_*)z \|_{L_{p,\nu}(I_j; X_0)} & \le c(p,\nu) | I_j |^{\nu-1/p}  \| \sF^\prime (z_*)\|_{\cL (\EE_{1,\mu}(J_T), BU\!C(J_T ;X_0))}  \|z\|_{\EE_{1,\nu}(I_j)} \\
&\le C(p,\nu, z_*) \, | I_j |^{\nu-1/p}\, \|z\|_{\EE_{1,\nu}(I_j)},
\qquad z\in {_0}\EE_{1,\nu}(I_j).
\end{aligned}
\end{equation*}
Combining all the steps, we have shown that given any $\eta>0$, there exist a number $\delta>0$ such that
$$
\| \sR^1_j(\cdot) z\|_{\EE_{0,\nu}(I_j)}\le \eta \|z\|_{\EE_{1,\nu}(I_j)},
$$
provided $| I_j |\le \delta $,
and the proof of the proposition is now complete.
\end{proof}
\begin{proposition}
\label{pro: right-inverse}
Assume  that \eqref{permittivity} and \eqref{mu-condition} hold.
For $z=(\bc, v,d)\in X_{\gamma ,\mu }$ let
\begin{equation*}
\begin{aligned}
&\cB_k(z)= B_k (\bc, d) c_k = \nu \cdot  {\rm tr }_{\pa\Omega} (\cD_k \nabla c_k) + b_{k,\pa \Omega}(\bc, d) {\rm tr}_{\pa \Omega} c_k, \\
&\cB(z)= (\cB_1(z),\ldots, \cB_N(z)),
\end{aligned}
\end{equation*}
where $\bc =(c_1,\ldots, c_N)$ and where the terms $b_{k ,\pa\Omega}(\bc, d)$ are defined in~\eqref{Phi-boundary}.
Then
$$\cB \in C^1 (X_{\gamma,\mu}, Y_{\gamma, \mu})\quad \text{with}\quad Y_{\gamma,\mu}= W^{2(\mu -1/p)-(1+1/p)}_p(\pa\Omega; \R^N).$$
Moreover, for each $\wt z\in X_{\gamma,\mu}$ there exists a right-inverse
$$\cR(\wt z)\in \cL( Y_{\gamma,\mu}, X_{\gamma,\mu}) \quad \text{for} \quad \cB^\prime (\wt z)\in \cL (X_{\gamma,\mu}, Y_{\gamma,\mu}).$$
\end{proposition}
\begin{proof}
It follows from Lemma~\ref{elliptic-tilde} that $\cB\in C^1(X_{\gamma,\mu}, Y_{\gamma,\mu})$. One readily verifies that
$$
\cB^\prime_k(\wt \bc, \wt d)(\bc, d) = B_k(\wt\bc, \wt d) c_k
+ [D_1 b_{k,\pa\Omega} (\wt \bc, \wt d) \bc ]{\rm tr}_{\pa \Omega} \wt c_k  + [D_2 b_{k,\pa\Omega} (\wt \bc , \wt d)d] {\rm tr}_{\pa \Omega} \wt c_k ,
$$
where $D_1b_k$ and $D_2b_k$ are the derivatives of $b_k$ with respect to $\bc$ and $d$, respectively.
Setting 
\begin{equation*}
\begin{aligned} 
& \fB(\wt \bc, \wt d) \bc =
\begin{bmatrix*}[l]
 & \!\!\!\!\!  B_1(\wt \bc, \wt d) c_1 
 + [D_1b_{1,\pa \Omega}(\wt \bc, \wt d) \bc ]{\rm tr}_{\pa \Omega}\wt c_1 \\
 & \qquad \vdots \\
 & \!\!\!\!\! B_N(\wt \bc, \wt d) c_N 
 +  [D_1b_{N,\pa \Omega}(\wt \bc, \wt d) \bc ]{\rm tr}_{\pa \Omega} \wt c_N \\
 \end{bmatrix*}\!\!,  \ 
 \fD(\wt \bc, \wt d)d=
\begin{bmatrix*}[l]
& \!\!\!\!\! [D_2 b_{1,\pa\Omega} (\wt \bc, \wt d)d]{\rm tr}_{\pa\Omega} \wt c_1 \\
& \qquad  \vdots \\
&  \!\!\!\!\! [D_2 b_{N,\pa\Omega} (\wt \bc, \wt d)d]{\rm tr}_{\pa\Omega} \wt c_N \\
\end{bmatrix*}
\end{aligned}
\end{equation*}
we have
\begin{equation}
\label{B+D}
\cB^\prime (\wt \bc,\wt d)(\bc, d)=  \fB(\wt \bc, \wt d) \bc + \fD(\wt \bc, \wt d)d.
\end{equation}
We note that $\fB(\wt \bc, \wt d) $ is a boundary operator of order one acting on $\bc$, while
$ \fD(\wt \bc, \wt d)$ is a zero order boundary operator acting on d.

\medskip\noindent
Let $\wt z= (\wt \bc, \wt d, \wt v)\in X_{\gamma,\mu}$ and $g\in Y_{\gamma,\mu }$ be given. We then set
\begin{equation}
\label{right-inverse}
 \cR(\wt z) g =(\bc, d, v) := ( {\mathfrak R}^{\bc}(\wt \bc, \wt d ) g, 0,0)\in  X_{\gamma,\mu}\
 \end{equation} 
where
${\mathfrak R}^{\bc } (\wt \bc, \wt d)\in \cL(Y_{\gamma, \mu},W^{2(\mu-1/p)}_p (\Omega; \R^N))$ is a a right-inverse for the boundary operator $ \fB(\wt \bc, \wt d).$
For the existence of ${\mathfrak R} ^{\bc}(\wt z) $ we refer to \cite[Proposition 2.5.1]{Mey10}.
It follows from \eqref{B+D} and \eqref{right-inverse} that
\begin{equation*}
\begin{aligned}
\cB^\prime (\wt \bc, \wt d) \cR (\wt \bc, \wt d)g 
= \fB(\wt \bc, \wt d) {\mathfrak R}^{\bc}(\tilde z) g + \fD(\wt \bc, \wt d) 0 = g 
\end{aligned}
\end{equation*}
for each $g\in Y_{\gamma ,\mu}$.
This shows that $\cR(\wt z)$ is a right inverse for $\cB^\prime(\wt z)$.
\end{proof}

\section{Proof of the energy estimate}
\label{Appendix B}
Here we follow the strategy used in \cite{FRSZ21}, where a similar situation is considered (though with periodic boundary conditions and a modified equation for the director $d$).
For the reader's convenience, we include a complete proof.
For brevity, we omit the integration symbol $dx$ throughout the proof.

\medskip\noindent
We know from Theorem~\ref{thm: main} that $z(\cdot, z^0)$ is smooth in time on $(0, T_+(z^0))$.  The first line in  $\eqref{main-PDE}$ then shows that 
$\Phi$ is smooth in time as well. One then verifies that the solution has enough regularity to carry out the computations below.

\medskip\noindent
Multiplying $\eqref{main-PDE}_1$ by $\mu_k$, integrating by parts using the boundary conditions \eqref{main-PDE}$_2$ and summing in $c_k$, we obtain 
\begin{equation*}
\frac{d}{dt}\left(\int_{\Omega}\sum_k c_k(\ln c_k-1)  \right)+\int_{\Omega}(\pa_t\rho)\Phi  + \int_{\Omega}(v\cdot\na \rho)\Phi  
+\alpha\sum_k \int_{\Omega}c_k|\na \mu_k|^2  \le 0.
\end{equation*}
Here we set $\rho =\sum z_k c_k$ and also employed  ${\rm div}\,v=0$ and $v=0$ on $\pa\Omega$ to conclude that  
$$
\int_\Omega (v\cdot \nabla c_k ) \ln c_k  = \int_\Omega v\cdot \nabla (c_k(\ln c_k -1))   =0.
$$
Using that $\int_\Omega c_k$ is constant in time, we then have
\begin{equation}
\label{npenergy}
 \frac{d}{dt}\left(\int_{\Omega}\sum_k c_k\ln c_k  \right)+\int_{\Omega}(\pa_t\rho)\Phi   
 +\int_{\Omega}(v\cdot\na \rho)\Phi  +\alpha\sum_k \int_{\Omega}c_k|\na \mu_k|^2  \le 0.
\end{equation}
Next, multiplying \eqref{main-PDE}$_2$ by $\pa_t\Phi$ and integrating by parts, we obtain (noting that $\pa_t\Phi_{|\pa\Omega}=0$ and $\ve$ is symmetric),
\begin{equation}
\label{poisson2}
    \frac{1}{2} \frac{d}{dt}\int_{\Omega}\ve(d)\na\Phi\cdot\na\Phi  =\int_{\Omega}\rho\pa_t\Phi   +\frac{\ve_a}{2}\int_{\Omega}\pa_t(d\otimes d)\na\Phi\cdot\na\Phi  .
\end{equation}
Returning to \eqref{npenergy} and using \eqref{poisson2} 
we have
\begin{equation}
\label{npenergy2}
\begin{aligned}
& \frac{d}{dt}\int_{\Omega}\sum_k c_k\ln c_k   +\frac{d}{dt}\int_{\Omega}\rho\Phi  - \frac{1}{2}\frac{d}{dt}\int_{\Omega}\ve(d)\na\Phi\cdot\na\Phi  \\
&\quad +\frac{\ve_a}{2}\int_{\Omega}\pa_t(d\otimes d)\na\Phi\cdot\na\Phi  
+\int_{\Omega}(v\cdot\na \rho)\Phi   +\alpha\sum_k \int_{\Omega}c_k|\na \mu_k|^2  \le 0. 
\end{aligned}
\end{equation}
Multiplying \eqref{main-PDE}$_3$ by $\Phi$ and integrating by parts yields
$$\int_{\Omega}\ve(d)\na\Phi\cdot\na\Phi=\int_{\Omega}\rho\Phi$$
and substituting this relation into \eqref{npenergy2} results in
\begin{equation*}
\begin{aligned}
    &\frac{d}{dt}\left(\int_{\Omega}\sum_k c_k\ln c_k+\frac{1}{2}\int_{\Omega}\ve(d)\na\Phi\cdot\na\Phi\right)+\frac{\ve_a}{2}\int_{\Omega}\pa_t(d\otimes d)\na\Phi\cdot\na\Phi\\
&\qquad +\int_{\Omega}(v\cdot\na \rho)\Phi+\alpha\sum_k \int_{\Omega}c_k|\na \mu_k|^2\le 0 .
\end{aligned}
\end{equation*}
Taking the inner product of $\eqref{main-PDE}_5$ with $v$ and integrating by parts, we obtain
\begin{equation}
\label{nsenergy}
\begin{aligned}
    &\frac{1}{2}\frac{d}{dt}\|v\|_{L^2}^2+\alpha_4 \int_\Omega | D(v)|^2 \\
    & = \int_{\Omega}(\na d\odot\na d):\na v  -\int_{\Omega}  (\na\Phi\otimes\na\Phi)\ve(d):\na v   \\
    & - \!\int_{\Omega}\!(\alpha_1(D(v)d\cdot d)d\otimes d+\alpha_2 \od\otimes d+\alpha_3 d\otimes \od+\alpha_5 D(v)d\otimes d+\alpha_6 d\otimes D(v)d)\!:\!\na v .
\end{aligned}
\end{equation}
Next, taking the inner product of the equation for $d$ in $\eqref{main-PDE}_8$ with $\dot{d}=\pa_t d+v\cdot\na d$ yields
\begin{equation}
\label{direnergy}
    \begin{aligned}
        &\int_{\Omega}(\gamma_1\od+\gamma_2D(v)d)\cdot \dot{d} +\frac{1}{2}\frac{d}{dt}\int_\Omega |\na d|^2 +\int_{\Omega}\na d : \na(v\cdot\na d) \\
        & \quad =\frac{\ve_a}{2}\int_{\Omega}(\na\Phi\otimes\na\Phi):\pa_t(d\otimes d)  +\ve_a\int_{\Omega}((\na\Phi\otimes\na\Phi)d)\cdot(v\cdot\na d) .
    \end{aligned}
\end{equation}
In deriving \eqref{direnergy} we used the fact that $|d|=1$ and \eqref{POmega} to conclude that 
$$
\gamma_2 P(d) D(v) d \cdot \dot d =  \gamma_2D(v) d \cdot \dot d,  \quad
\ve_a(P(d) \nabla (\Phi \otimes \nabla \Phi)d  \cdot  \dot d =  \ve_a\nabla (\Phi \otimes \nabla \Phi)d  \cdot  \dot d 
$$ 
and $| \nabla d|^2d \cdot \dot d =\frac{1}{2}| \nabla d|^2 \big[ \partial_t |d| ^2 + v\cdot \nabla |d|^2 \big ] =0$.  Moreover, we used the boundary condition 
$\partial_\nu d=0$ to obtain
$$
\int_\Omega \Delta d  \cdot \partial_t d  = \int_{\pa\Omega} \partial_\nu d \cdot \partial_t d - \frac{1}{2}\frac{d}{dt} \int_\Omega | \nabla d |^2   
=   - \frac{1}{2}\frac{d}{dt} \int_\Omega | \nabla d |^2    
$$
and
$$
\int_\Omega \Delta d \cdot (v\cdot \nabla d )  
= \int_{\pa\Omega} \partial_\nu d \cdot (v\cdot \nabla d) - \int_\Omega \nabla d: \nabla (v\cdot \nabla d)  
= - \int_\Omega \nabla d: \nabla (v\cdot \nabla d) . 
$$
The following observation will be very useful to proceed with the energy estimate.
\begin{lemma}
It holds that
\begin{equation*}
\begin{aligned}
&(\gamma_1\od+\gamma_2D(v)d)\cdot \dot{d} \\
&\quad + \left(\alpha_1(D(v)d\cdot d)d\otimes d+\alpha_2 \od\otimes d+\alpha_3 d\otimes \od+\alpha_5 D(v)d\otimes d+\alpha_6 d\otimes D(v)d):\na v\right) \\
& \quad  = \left(\alpha_1(d\cdot D(v)d)^2+(\gamma_2+\alpha_2+\alpha_3)(\od\cdot D(v)d)+(\alpha_5+\alpha_6)|D(v)d|^2+\gamma_1|\od|^2\right).
\end{aligned}
\end{equation*}
Note that the expression in the second line appears in \eqref{nsenergy}.
\end{lemma}
\begin{proof}
We calculate the sum piece by piece. Also, for brevity, we omit the dependence of $D$ and $\Omega$ on $v$. We first note that
\begin{equation}
\label{AA}
\begin{aligned}
        (\gamma_1\od+\gamma_2Dd)\cdot\dot{d} &=(\gamma_1\od+\gamma_2Dd)\cdot(\od+\Omega d)\\
        &=\gamma_1|\mathring d|^2+\gamma_2(\od\cdot Dd)+\gamma_2(Dd\cdot\Omega d)+\gamma_1(\od\cdot\Omega d).
    \end{aligned}
    \end{equation}
We have
\begin{equation}
  \label{ABa}
    \begin{aligned}
        \alpha_1(Dd\cdot d)(d\otimes d):\na v &=\alpha_1(Dd\cdot d)(d\otimes d):(D+\Omega)\\
        & =\alpha_1(Dd\cdot d)(d\otimes d):D\\
        & =\alpha_1(Dd\cdot d)^2,
    \end{aligned}
    \end{equation}
 where the second equality follows from the fact that the Frobenius inner product between a symmetric matrix and a skew symmetric matrix is 0.
   
\medskip\noindent Observe that
     $$(\alpha_2\od\otimes d+\alpha_3 d\otimes \od):\na v=(\alpha_2\od\otimes d+\alpha_3 d\otimes \od):(D+\Omega).$$
     Considering all the terms on the right side separately results in 
    \begin{itemize}
    \vspace{1mm}
        \item $\alpha_2(\od\otimes d):D=\alpha_2  \od_i d_j D_{ij}   =\alpha_2 (\od\cdot Dd)$.
       \vspace{1mm}
        \item $\alpha_3 (d\otimes\od):D=\alpha_3 d_i \od_j D_{ij}=\alpha_3 d_j \od_iD_{ji}=\alpha_3 d_j \od_iD_{ij}=\alpha_3(\od\cdot Dd)$, where the third equality follows from the symmetry of $D$.
       \vspace{1mm} 
        \item $\alpha_2(\od\otimes d):\Omega=\alpha_2 \od_id_j\Omega_{ij}= \alpha_2(\od\cdot \Omega d)$.
        \vspace{1mm} 
        \item $\alpha_3(d\otimes \od):\Omega = \alpha_3 d_i\od_j\Omega_{ij}=\alpha_3 d_j \od_i\Omega_{ji}=-\alpha_3 d_j\od_i\Omega_{ij}=-\alpha_3(\od\cdot\Omega d)$ where the third equality follows from the skew symmetry of $\Omega$.
    \end{itemize}
    \vspace{1mm}
    Adding the terms above and using \eqref{gamma-alpha} yields
    \begin{equation}
    \label{AC}
    (\alpha_2\od\otimes d+\alpha_3 d\otimes \od):\na v =(\alpha_2+\alpha_3)(\od\cdot Dd)-\gamma_1(\od\cdot\Omega d).
    \end{equation}
  Finally, we consider the term
   $$(\alpha_5 (Dd\otimes d)+\alpha_6 (d\otimes Dd)):\na v=(\alpha_5 (Dd\otimes d)+\alpha_6 (d\otimes Dd)):(D+\Omega).$$
    Here we get
    \begin{itemize}
    \vspace{1mm}
        \item $\alpha_5 (Dd\otimes d):D=\alpha_5 D_{ij}d_j d_k D_{ik}=\alpha_5|Dd|^2$.
     \vspace{1mm}     
        \item $\alpha_6(d\otimes Dd):D=\alpha_6 d_i D_{jk}d_k D_{ij}=\alpha_6  D_{jk}d_k D_{ji}d_i=\alpha_6|Dd|^2 $.
       \vspace{1mm}
        \item $\alpha_5(Dd\otimes d):\Omega=\alpha_5 D_{ij}d_j d_k\Omega_{ik}=\alpha_5 (Dd\cdot \Omega d)$.
        \vspace{1mm}
       \item $\alpha_6(d\otimes Dd):\Omega = \alpha_6 d_iD_{jk}d_k\Omega_{ij}=-\alpha_6 d_iD_{jk}d_k\Omega_{ji}=-\alpha_6(Dd\cdot\Omega d)$.
    \end{itemize}
    Adding the terms above, we get with \eqref{gamma-alpha}
    \begin{equation}
    \label{AD}
    (\alpha_5 (Dd\otimes d)+\alpha_6 (d\otimes Dd)):\na v = \beta |Dd|^2-\gamma_2(Dd\cdot\Omega d),
    \end{equation}
where $\beta = \alpha_5 + \alpha_6$. Adding the terms in \eqref{AA}-\eqref{AD} yields the assertion.
 \end{proof}
Thus, so far, we have
\begin{equation}\label{npenergy4}
    \begin{aligned}
        &\frac{d}{dt}\left(\int_{\Omega}\sum_k c_k\ln c_k   +\frac{1}{2}\int_{\Omega}\big(\ve(d)\na\Phi\cdot\na\Phi+|\na d|^2+|v|^2) \right) +\alpha_4\int_\Omega |D(v)|^2  \\
        &+\int_{\Omega}\left(\alpha_1(d\cdot D(v)d)^2+(\gamma_2+\alpha_2+\alpha_3)(\od\cdot D(v)d)+ \beta |D(v)d|^2+\gamma_1|\od|^2\right)\\
&+\int_{\Omega}(v\cdot\na \rho)\Phi +\alpha\sum_k \int_{\Omega}c_k|\na \mu_k|^2  \\
&\le \ve_a\int_{\Omega}((\na\Phi\otimes\na\Phi)d)\cdot(v\cdot\na d)  -\int_{\Omega}  (\na\Phi\otimes\na\Phi)\ve(d):\na v ,
    \end{aligned}
\end{equation}
where we note the cancellations
$$
\frac{\ve_a}{2}\int_{\Omega}\pa_t(d\otimes d)\na\Phi\cdot\na\Phi   - \frac{\ve_a}{2}\int_{\Omega}(\na\Phi\otimes\na\Phi):\pa_t(d\otimes d) =0,
$$
and (using ${\rm div}\, v=0$, $v=0$ on $\pa\Omega$)
$$
\int_{\Omega}\na d:\na(v\cdot\na d)  -\int_{\Omega}(\na d\odot\na d):\na v   =0.
$$
Next, we multiply $\eqref{main-PDE}_3$ by $v\cdot\na\Phi$ and integrate by parts to obtain
$$
\int_{\Omega}\ve(d)\na\Phi\cdot\na(v\cdot\na\Phi)   =\int_{\Omega}\rho v\cdot\na\Phi   =-\int_{\Omega}(v\cdot\na\rho)\Phi   .
$$
Substituting this relation into \eqref{npenergy4} results in 
\begin{equation}
\label{npenergy5}
    \begin{aligned}
        &\frac{d}{dt}\left(\int_{\Omega}\sum_k c_k\ln c_k  +\frac{1}{2}\int_{\Omega}\big(\ve(d)\na\Phi\cdot\na\Phi+|\na d|^2+ |v|^2)  \right)
        + \alpha_4 \int_\Omega |D(v)|^2  \\
        &+\int_{\Omega}\left(\alpha_1(d\cdot D(v)d)^2+(\gamma_2+\alpha_2+\alpha_3)(\od\cdot D(v)d)+\beta |D(v)d|^2+\gamma_1|\od|^2\right)\\
&-\int_{\Omega}\ve(d)\na\Phi\cdot\na(v\cdot\na\Phi) +\alpha\sum_k \int_{\Omega}c_k|\na \mu_k|^2  \\
&\le \ve_a\int_{\Omega}((\na\Phi\otimes\na\Phi)d)\cdot(v\cdot\na d) -\int_{\Omega}  (\na\Phi\otimes\na\Phi)\ve(d) :\na v. 
    \end{aligned}
\end{equation}
Let us observe that, using ${\rm div}\, v=0$ and $v|_{\pa\Omega}=0$,  we have
$$
\int_{\Omega}\na\Phi\cdot\na(v\cdot\na\Phi) =\int_{\Omega}(\na\Phi\otimes\na\Phi):\na v .
$$
Thus, from \eqref{npenergy5}, we obtain
\begin{equation}
\label{npenergy6}
    \begin{aligned}
       &\frac{d}{dt}\left(\int_{\Omega}\sum_k c_k\ln c_k  +\frac{1}{2}\int_{\Omega}\big(\ve(d)\na\Phi\cdot\na\Phi+|\na d|^2+ |v|^2)  \right)
        + \alpha_4 \int_\Omega |D(v)|^2  \\
        &+\int_{\Omega}\left(\alpha_1(d\cdot D(v)d)^2+(\gamma_2+\alpha_2+\alpha_3)(\od\cdot D(v)d)+\beta\, |D(v)d|^2+\gamma_1|\od|^2\right)\\
&-\ve_a\int_{\Omega}(d\otimes d)\na\Phi\cdot\na(v\cdot\na\Phi) +\alpha\sum_k \int_{\Omega}c_k|\na \mu_k|^2 \\
&\le \ve_a\int_{\Omega}((\na\Phi\otimes\na\Phi)d)\cdot(v\cdot\na d) -\ve_a\int_{\Omega}  (\na\Phi\otimes\na\Phi)(d\otimes d):\na v.
    \end{aligned}
\end{equation}
It remains to show that 
\begin{equation}
\label{identity}
\begin{aligned}
    & \int_{\Omega}(\na\Phi\otimes\na\Phi) (d\otimes d):\na v   \\
    &\qquad =\int_{\Omega}((\na\Phi\otimes\na\Phi)d)\cdot(v\cdot\na d)   +\int_{\Omega}(d\otimes d)\na\Phi\cdot\na(v\cdot\na\Phi)  .
\end{aligned}
\end{equation}
For this, we first note that
\begin{align*}
    \int_{\Omega} (d\otimes d) \na\Phi\cdot\na(v\cdot\na\Phi)  & =\int_{\Omega}d_i d_j\pa_j\Phi\pa_i(v_k\pa_k\Phi) \\
    & =\int_{\Omega}d_id_j\pa_j\Phi\pa_iv_k\pa_k\Phi  +\int_{\Omega}d_id_j\pa_j\Phi v_k\pa_k\pa_i\Phi   \\
    & =\int_{\Omega}\pa_k\Phi\pa_j\Phi d_j d_i \pa_iv_k   +\int_{\Omega}d_id_j\pa_j\Phi v_k\pa_k\pa_i\Phi  \\
    & =\int_\Omega (\na\Phi\otimes\na\Phi)(d\otimes d):\na v  +\int_\Omega d_id_j\pa_j \Phi v_k\pa_k\pa_i\Phi  .
\end{align*}
Thus, \eqref{identity} reduces to
\begin{align}
\label{identity2}
    0=\int_\Omega ((\na\Phi\otimes\na\Phi)d)\cdot(v\cdot\na d)  +\int_\Omega d_id_j\pa_j\Phi v_k\pa_k\pa_i\Phi   .
\end{align}
We compute, integrating by parts, using ${\rm div}\, v = 0$, and permuting indices (and noting that no boundary terms occur because $v_{|\pa\Omega}=0$):
\begin{align*}
    \int_\Omega d_id_j\pa_j\Phi v_k\pa_k\pa_i\Phi& =-\int_\Omega \pa_k d_i d_j\pa_j\Phi v_k\pa_i \Phi -\int_\Omega d_i \pa_k d_j\pa_j\Phi v_k\pa_i\Phi-\int_\Omega d_i d_j\pa_j\pa_k\Phi v_k\pa_i \Phi\\
    & =-\int_\Omega \pa_k d_i d_j\pa_j\Phi v_k\pa_i \Phi -\int_\Omega d_j \pa_k d_i\pa_i\Phi v_k\pa_j\Phi-\int_\Omega d_i d_j\pa_i\pa_k\Phi v_k\pa_j \Phi.
\end{align*}
It follows from the above that
\begin{align*}
    2\int_\Omega d_id_j\pa_j\Phi v_k\pa_k\pa_i\Phi =-2\int_\Omega \pa_k d_i d_j\pa_j\Phi v_k\pa_i \Phi , 
\end{align*}
which is equivalent to \eqref{identity2}. The assertion of Proposition~\ref{energy-est} follows now from \eqref{npenergy6} and \eqref{identity}.
\qquad $\square$ 

\bigskip\noindent
{\bf Conflict of interest:} The authors assert that there is no conflict of interest to declare.

\medskip\noindent
{\bf Data availability statement:}
Data sharing is not applicable as no datasets were generated or analyzed for the manuscript.

\bigskip\noindent

\end{document}